\documentclass{amsart}
\usepackage{amssymb,amsmath,amsthm}
\usepackage[foot]{amsaddr}
\title{Exact Completion of Path Categories and Algebraic Set Theory \\[1ex]
  \footnotesize\mdseries Part I: Exact Completion of Path Categories}
\author{Benno van den Berg$^1$}
\address{${}^1$ Institute for Logic, Language and Computation (ILLC), University of Amsterdam, P.O. Box 94242, 1090 GE Amsterdam, the Netherlands. E-mail: bennovdberg@gmail.com.}
\author{Ieke Moerdijk$^2$}
\address{${}^2$ Mathematical Institute, Utrecht University, P.O. Box 80010, 3508 TA Utrecht, the Netherlands. E-mail: I.Moerdijk@uu.nl.}
\date{\today}
\usepackage{ast2}
\usepackage[margin=1.7in]{geometry}

\begin{document}

\begin{abstract}
We introduce the notion of a ``category with path objects'', as a slight strengthening of Kenneth Brown's classical notion of a ``category of fibrant objects''. We develop the basic properties of such a category and its associated homotopy category. Subsequently, we show how the exact completion of this homotopy category can be obtained as the homotopy category associated to a larger category with path objects, obtained by freely adjoining certain homotopy quotients. In a second part of this paper, we will present an application to  models of constructive set theory. Although our work is partly motivated by recent developments in homotopy type theory, this paper is written purely in the language of homotopy theory and category theory, and we do not presuppose any familiarity with type theory on the side of the reader.
\end{abstract}

\maketitle

\section{Introduction}

The phrase ``path category'' in the title is short for ``category with path objects'' and refers to a modification of Kenneth Brown's notion of a category of fibrant objects \cite{brown73}, originally meant to axiomatise the homotopical properties of the category of simplicial sheaves on a topological space. Like categories of fibrant objects, path categories are categories equipped with classes of fibrations and weak equivalences, and as such they are closely related to Quillen's model categories which have an additional class of cofibrations \cite{quillen67,quillen69,hovey99}. Our modification of Brown's definition mainly consists in an additional axiom which in the language of Quillen model categories would amount to the condition that every object is cofibrant. One justification for this modification is that there still are plenty of examples. One source of examples is provided by taking the fibrant objects in a model category in which all objects are cofibrant, such as the category of simplicial sets, or the categories of simplicial sheaves equipped with the injective model structure. More generally, many model categories have the property that objects over a cofibrant object are automatically cofibrant. For example, this holds for familiar model category structures for simplicial sets with the action of a fixed group, for dendroidal sets, and for many more. In such a model category, the fibrations and weak equivalences between objects which are both fibrant and cofibrant satisfy our modification of Brown's axioms.

Another justification, and in fact our main motivation, for this modification of Brown's axioms is that these modified axioms are satisfied by the syntactic category constructed out of a type theory \cite{avigadetal15,gambinogarner08}. Thus, our work builds on the recently discovered interpretation of Martin-L\"of type theory in Quillen model categories \cite{awodeywarren09}. This interpretation has been extended by Voevodsky to an interpretation of the Calculus of Constructions in the category of simplicial sets \cite{kapulkinetal12} (see also \cite{bergmoerdijk15,cisinski14,gepnerkock17,shulman15a,shulman15b}).

In addition, our work is relevant for constructive set theory. Aczel has provided an interpretation of the language of set theory in a type theory with a suitable universe \cite{aczel78}, and the question arises whether it is possible to construct models of set theory out of certain path categories. We will turn to this question in Part II of this paper.

The precise contents of this paper are as follows. In Section 2 we introduce the notion of a path category and verify that many familiar constructions from homotopy theory can be performed in such path categories and retain their expected properties. It is necessary for what follows to perform this verification, but there is very little originality in it. An exception is perhaps formed by our construction of suitable path objects carrying a connection structure as in \reftheo{existenceconnections} and our statement concerning the existence of diagonal fillers which are half strict, half up-to-homotopy, as in \reftheo{allgood} below. We single out these two properties here also because they play an important r\^ole in later parts of the paper.

In Section 3 we will introduce a notion of ``homotopy exact completion'' for such path categories, a new category obtained by freely adjoining certain homotopy quotients. For ``trivial'' path categories in which every map is a fibration and only isomorphisms are weak equivalences this notion of homotopy exact completion coincides with the ordinary notion of exact completion, well known from category theory (see \cite{carboni95,carboniceliamagno82,carbonivitale98}). In case the path category is obtained from the syntax of type theory this coincides with what is known as the setoids construction (see \cite{bartheetal03}). Indeed, the type-theorist can think of our work as a categorical analysis of this construction informed by the homotopy-theoretic interpretation of type theory. The main result in Section 3 shows that the exact completion of the homotopy category of a path category \ct{C} is itself a homotopy category of another path category which we call ${\rm Ex}(\ct{C})$, see \refprop{conntoordintheory} and \reftheo{hexcatoffibrantobj} below.

In Section 4 we show that if \ct{C} has homotopy sums which are, in a suitable sense, stable and disjoint, then the homotopy exact completion is a pretopos (see \reftheo{hsumsunderhex}). We will also show that the homotopy exact completion has a natural numbers object if \ct{C} has what we will call a homotopy natural numbers object.

Finally, in Section 5 we will show that the homotopy exact completion improves the properties of the original category in that it will satisfy certain extensionality principles even when the original category does not. This is analogous to what happens for ordinary exact completions: the ordinary exact completion $\ct{C}'$ of a category \ct{C} will be locally cartesian closed (that is, will have internal homs in every slice) whenever \ct{C} has this property in a weak form, where weak is meant to indicate that one weakens the usual universal property of the internal hom by dropping the uniqueness requirement, only keeping existence (see \cite{carbonirosolini00}). In the same vein we show in Section 5 that if a path category has weak homotopy $\Pi$-types (i.e. weak fibrewise up-to-homotopy internal homs) then its exact completion has exponentials in every slice. In type-theoretic terms this means that the homotopy exact completion will always satisfy a form of function extensionality; something similar holds for the path category ${\rm Ex}(\ct{C})$.

At this point it is probably good to add a few words about our approach and how it relates to some of the work that is currently being done at the interface of type theory and homotopy theory. First of all, we take a resolutely categorical approach; in particular, no knowledge of the syntax of type theory is required to understand this paper. As a result, we expect our paper to be readable by homotopy theorists.

Moreover, despite being inspired by homotopy type theory, the additions to Martin-L\"of type theory suggested by its homotopy-theoretic interpretation play no r\^ole in this paper. In particular, we will not use univalence, higher-inductive types or even function extensionality. Indeed, all the definitions and theorems have been formulated in such a way that they will apply to the syntactic category of (pure, intensional) Martin-L\"of type theory. In fact, we expect our definitions remain applicable to the syntactic category of type theory even when all its computation rules are formulated as propositional equalities. One of the authors of this paper has verified this in detail for the identity types (see \cite{vandenberg16}), but we firmly believe that it applies to all type constructors. This idea has guided us in setting up many of the definitions of this paper. This includes, for example, the definition of a (weak) homotopy $\Pi$-type as in \refdefi{hPitypes} below.

In addition to the reasons already mentioned above, these considerations have determined our choice to work in the setting of path categories. As said, our path categories are related to categories of fibrant objects \emph{\`a la} Brown, or fibration categories as they have been called by other authors. Structures similar to fibration categories or their duals have been studied by Baues \cite{baues89} and Waldhausen \cite{waldhausen85}, for homotopy-theoretic purposes. For a survey and many basic properties, we refer to \cite{radulescubanu09}.

More recently, several authors have also considered such axiomatisations in order to investigate the relation between homotopy theory and type theory. For instance, Joyal (unpublished) and Shulman \cite{shulman15b} have considered axiomatisations in terms of a weak factorisation system for fibrations and acyclic cofibrations, a set-up which is somewhat stronger than ours. In our setting we do not have such a weak factorisation system, and the lifting properties that we derive in our path categories yield diagonals that make lower triangles strictly commutative, while upper triangles need only commute up to (fibrewise) homotopy. Our reasons for deviating from Joyal and Shulman are that in the setting of the weak rules for the identity types such weak liftings seem to be the best possible; in addition, the category ${\rm Ex}(\ct{C})$ only seems to be a path category in our sense, even when \ct{C} is a type-theoretic fibration category in the sense of Shulman.

In this paper we have not entered into any $\infty$-categorical aspects. For readers interested in the use of fibration categories in $\infty$-category theory and its relation to type theory, we refer to the work of Kapulkin and Szumi\l{}o \cite{kapulkin15,kapulkinszumilo17,szumilo14}.

\subsection{Acknowledgments} The writing of this paper took place in various stages, and versions of the results that we describe here were presented at various occasions. We are grateful to the organisers of the Homotopy Type Theory Workshop in Oxford in 2014, TACL 2015 in Salerno and the minisymposium on Homotopy Type Theory and Univalent Foundations at the Jahrestagung der DMV 2015 in Hamburg for giving us the opportunity to present earlier versions of parts of this paper. We owe a special debt to the Newton Institute for Mathematical Sciences in Cambridge and the Max Planck Institute in Bonn. The first author was a visiting fellow at the Newton Institute in the programme ``Mathematical, Foundational and Computational Aspects of the Higher Infinite (HIF)'' in Fall 2015, while both authors participated in the ``Program on Higher Structures in Geometry and Physics'' at the Max Planck Institute in 2016.  At both institutes various parts of this paper were written and presented. We would also like to thank Chris Kapulkin for useful bibliographic advice and the referees for a careful reading of the manuscript. Finally, we are grateful to Peter Lumsdaine and one of the referees for pointing out an error in an earlier version of this paper. 

\section{Path categories}

\subsection{Axioms}
Throughout this paper we work with path categories, a modification of Brown's notion of a category of fibrant objects \cite{brown73}. We will start by recalling Brown's definition.

The basic structure is that of a category \ct{C} together with two classes of maps in \ct{C} called the \emph{weak equivalences} and the \emph{fibrations}, respectively. Morphisms which belong to both classes of maps will be called \emph{acylic fibrations}. A \emph{path object} on an object $B$ is a factorisation of the diagonal $\Delta_B: B \to B \times B$ as a weak equivalence $r: B \to PB$ followed by a fibration $(s, t): PB \to B \times B$.

\begin{defi}{catfibrobj} \cite{brown73} The category \ct{C} is called a \emph{category of fibrant objects} if the following axioms are satisfied:
\begin{enumerate}
\item[(1)] Fibrations are closed under composition.
\item[(2)] The pullback of a fibration along any other map exists and is again a fibration.
\item[(3)] The pullback of an acylic fibration along any other map is again an acyclic fibration.
\item[(4$'$)] Weak equivalences satisfy 2-out-of-3: if $gf = h$ and two of $f, g, h$ are weak equivalences then so is the third.
\item[(5$'$)] Isomorphisms are acyclic fibrations.
\item[(6)] For any object $B$ there is a path object $PB$ (not necessarily functorial in $B$).
\item[(7)] \ct{C} has a terminal object $1$ and every map $X \to 1$ to the terminal object is a fibration.
\end{enumerate}
\end{defi}

We make two modifications to Brown's definition, the first of which is relatively minor. Instead of the more familiar 2-out-of-3 property we demand that the weak equivalences satisfy 2-out-of-6:
\begin{enumerate}
\item[(4)] Weak equivalences satisfy 2-out-of-6: if $f: A \to B$, $g: B \to C$, $h: C \to D$ are three composable maps and both $gf$ and $hg$ are weak equivalences, then so are $f,g,h$ and $hgf$.
\end{enumerate}
It is not hard to see that this implies 2-out-of-3.  We have decided to stick with the stronger property, as it is something which is both useful and true in all the examples we are interested in. (See also \refrema{on2outof6} below.)

A more substantial change is that we will add an axiom saying that every acyclic fibration has a section (this is sometimes expressed by saying that ``every object is cofibrant''). To be precise, we will modify (5$'$) to:
\begin{enumerate}
\item[(5)] Isomorphisms are acyclic fibrations and every acyclic fibration has a section.
\end{enumerate}
As discussed in the introduction, one reason we have made this change is that it is satisfied in the syntactic category associated to type theory \cite{avigadetal15} and in many situations occurring in homotopy theory. In fact, axiom (5) will be used throughout this paper and in this section we will investigate, somewhat systematically, the consequences of this axiom.

To summarise:
\begin{defi}{pathcat} The category \ct{C} will be called a \emph{category with path objects}, or a \emph{path category} for short, if the following axioms are satisfied:
\begin{enumerate}
\item[(1)] Fibrations are closed under composition.
\item[(2)] The pullback of a fibration along any other map exists and is again a fibration.
\item[(3)] The pullback of an acylic fibration along any other map is again an acyclic fibration.
\item[(4)] Weak equivalences satisfy 2-out-of-6: if $f: A \to B$, $g: B \to C$, $h: C \to D$ are three composable maps and both $gf$ and $hg$ are weak equivalences, then so are $f,g,h$ and $hgf$.
\item[(5)] Isomorphisms are acyclic fibrations and every acyclic fibration has a section.
\item[(6)] For any object $B$ there is a path object $PB$ (not necessarily functorial in $B$).
\item[(7)] \ct{C} has a terminal object $1$ and every map $X \to 1$ to the terminal object is a fibration.
\end{enumerate}
\end{defi}
We have chosen the name path category because its homotopy category is completely determined by the path objects (as every object is ``cofibrant'').

Examples are:
\begin{enumerate}
\item The syntactic category associated to type theory \cite{avigadetal15}. In fact, to prove that the syntactic category is an example, it suffices to assume that the computation rule for the identity type holds only in a propositional form (see \cite{vandenberg16}).
\item Let \ct{M} be a Quillen model category. If every object is cofibrant in \ct{M}, then the full subcategory of fibrant objects in \ct{M} is a path category in our sense. More generally, if any object over a cofibrant object is also cofibrant, then the full subcategory of fibrant-cofibrant objects in \ct{M} is a path category.
\item In addition, there is the following trivial example: if \ct{C} is a category with finite limits, it can be considered as a path category in which every morphism is a fibration and only the isomorphisms are weak equivalences. By considering this trivial situation, it can be seen that our theory of the homotopy exact completion in the next section generalises the classical theory of exact completions of categories with finite limits.
\end{enumerate}

\subsection{Basic properties} We start off by making some basic observations about path categories, all of which are due to Brown in the context of categories of fibrant objects (\cite{brown73}; see also \cite{radulescubanu09}). First of all, note that the underlying category \ct{C} has finite products and all projection maps are fibrations. From this it follows that if $(f, g): P \to X \times X$ is a fibration, then so are $f$ and $g$.

\begin{prop}{factlemma}
In a path category any map $f: Y \to X$ factors as $f = p_fw_f$ where $p_f$ is a fibration and $w_f$ is a section of an acylic fibration (and hence a weak equivalence).
\end{prop}
\begin{proof}
This is proved on page 421 of \cite{brown73}. Since the factorisation will be important in what follows, we include the details here. First observe that if $PX$ is a path object for $X$ with weak equivalence $r: X \to PX$ and fibration $(s, t): PX \to X \times X$, then it follows from 2-out-of-3 for weak equivalences and $sr = tr = 1$ that both $s, t: PX \to X$ are acyclic fibrations. So for any map $f: Y \to X$ the following pullback
\diag{ P_f \ar[r]^{p_2} \ar[d]_{p_1} & PX \ar[d]^s \\
Y \ar[r]_f & X,}
exists with $p_1$ being an acyclic fibration. We set $w_f := (1, rf): Y \to P_f$ and $p_f := tp_2: P_f \to X$. Then $p_f w_f = f$ and $w_f$ is a section of $p_1$. Moreover, the following square
\diag{ P_f \ar[r]^{p_2} \ar[d]_{(p_1, p_f)} & PX \ar[d]^{(s, t)} \\
Y \times X \ar[r]_{f \times 1} & X \times X.}
is a pullback, so $(p_1,  p_f)$ is fibration, which implies that $p_f$ is a fibration as well.
\end{proof}

\begin{coro}{makingweeasy} Any weak equivalence $f: Y \to X$ factors as $f = p_fw_f$ where $p_f$ is an acylic fibration and $w_f$ is a section of an acyclic fibration.
\end{coro}

\begin{defi}{slicingforcofo} If \ct{C} is a path category and $A$ is any object in \ct{C} we can define a new path category $\ct{C}(A)$, as follows: its underlying category is the full subcategory of $\ct{C}/A$ whose objects are the fibrations with codomain $A$. This means that its objects are fibrations $X \to A$, while a morphism from $q: Y \to A$ to $p: X \to A$ is a map $f: Y \to X$ in \ct{C} such that $pf = q$; such a map $f$ is a fibration or a weak equivalence in $\ct{C}(A)$ precisely when it is a fibration or a weak equivalence in \ct{C}.
\end{defi}

Clearly, $\ct{C}(1) \cong \ct{C}$. Observe that for any $f: B \to A$ there is a pullback functor $f^*: \ct{C}(A) \to \ct{C}(B)$, since pullbacks of fibrations always exist and are again fibrations.

\begin{prop}{presbypb} For any morphism $f: B \to A$ the functor $f^*: \ct{C}(A) \to \ct{C}(B)$ preserves both fibrations and weak equivalences.
\end{prop}
\begin{proof}
This is proved on page 428 of \cite{brown73} and the proof method is often called Brown's Lemma. The idea is that Axiom 3 for path categories tells us that $f^*$ preserves acyclic fibrations. But then it follows from the previous corollary and 2-out-of-3 for weak equivalences that $f^*$ preserves weak equivalences as well.
\end{proof}

This proposition can be used to derive:
\begin{prop}{basechangeforwe}
The pullback of a weak equivalence $w: A' \to A$ along a fibration $p: B \to A$ is again a weak equivalence.
\end{prop}
\begin{proof}
See pages 428 and 429 of \cite{brown73}.
\end{proof}

\subsection{Homotopy} In any path category we can define an equivalence relation on the hom-sets: the homotopy relation.
\begin{defi}{homotopy}
Two parallel arrows $f, g: Y \to X$ are \emph{homotopic}, if there is a path object $PX$ for $X$ with fibration $(s, t): PX \to X \times X$ and a map $h: Y \to PX$ (the \emph{homotopy}) such that $f = sh$ and $g = th$.  In this case, we write $f \simeq g$, or $h: f \simeq g$ if we wish to stress the homotopy $h$.
\end{defi}
At present it is not clear that this definition is independent of the choice of path object $PX$, or that it defines an equivalence relation. In order to prove this, we use the following lemma, which is a consequence of (and indeed equivalent to) the axiom that every acyclic fibration has a section.

\begin{lemm}{lifting}
Suppose we are given a commutative square
\diag{ D \ar[r]^g \ar[d]_w & C \ar[d]^p \\
B \ar[r]_ f & A}
in which $w$ is a weak equivalence and $p$ is a fibration. Then there is a map $l: B \to C$ such that $p l = f$ (for convenience, we will call such a map a \emph{lower filler}).
\end{lemm}

\begin{proof}
Let $k: D \to B \times_A C$ be the map to the pullback with $p_1 k = w$ and $p_2 k = g$, and factor $k$ as $k = q i$ where $i$ is a weak equivalence and $q$ is a fibration. Then $p_1q$ is an acyclic fibration and hence has a section $a$. So if we put $l := p_2qa$, then $pl = pp_2qa = fp_1qa = f$, as desired.
\end{proof}
Just in passing we should note that a statement much stronger than \reflemm{lifting} is true, but that in order to state and prove it we need to develop a bit more theory (see \reftheo{allgood} below).

\begin{coro}{onpathobj}
If $PX$ is a path object for $X$ and $PY$ is a path object for $Y$ and $f: X \to Y$ is any morphism, then there is a map $Pf: PX \to PY$ such that
\diag{ PX \ar[d]_{(s,t)} \ar[r]^{Pf} & PY \ar[d]^{(s, t)} \\
X \times X \ar[r]_{f \times f} & Y \times Y}
commutes. In particular, if $PX$ and $P'X$ are two path objects for $X$ then there is a map $f: PX \to P'X$ which commutes with the source and target maps of $PX$ and $P'X$.
\end{coro}
\begin{proof}
Any lower filler in the diagram
\begin{align*}
\xymatrix{ X \ar[r]^{rf} \ar[d]_r & PY \ar[d]^{(s, t)} \\
PX \ar[r]_(.45){(fs, ft)} & Y \times Y   }
\end{align*} 
gives us the desired arrow.
\end{proof}
The second statement in the previous corollary implies that if two parallel maps $f, g: X \to Y$ are homotopic relative to one path object $PY$ on $Y$, then they are homotopic with respect to any path object on $Y$; so in the definition of the homotopy relation nothing depends on the choice of the path object.

In order to show that the homotopy relation is an equivalence relation, and indeed a congruence, we introduce the following definition, which will also prove useful later.
\begin{defi}{hequivalencerel} A fibration $p = (p_1, p_2): R \to X \times X$ is a \emph{homotopy equivalence relation}, if the following three conditions are satisfied:
\begin{enumerate}
\item There is a map $\rho: X \to R$ such that $p \rho = \Delta_X$.
\item There is a map $\sigma: R \to R$ such that $p_1 \sigma = p_2$ and $p_2 \sigma=p_1$.
\item For the pullback
\diag{ R \times_X R \ar[r]^(.55){q_2} \ar[d]_{q_1} & R \ar[d]^{p_1} \\
R \ar[r]_{p_2} & X }
there is a map $\tau: R \times_X R \to R$ such that $p_1q_1 = p_1 \tau$ and $p_2q_2 = p_2 \tau$.
\end{enumerate}
\end{defi}

\begin{prop}{pathobjheq}
If $PX$ is a path object with fibration $p = (s, t): PX \to X \times X$ and weak equivalence $r: X \to PX$, then $p$ is a homotopy equivalence relation.
\end{prop}
\begin{proof}
\begin{enumerate}
\item We put $\rho = r$.
\item The map $\sigma$ is obtained as a lower filler in:
\diag{ X \ar[d]_{r} \ar[r]^{r} & PX \ar[d]^{(s, t)} \\
PX \ar[r]_(.45){(t, s)} & X \times X.}
\item Let $\alpha$ be the unique map filling
\diag{ X \ar@{.>}[dr]_{\alpha} \ar@/^/[drr]^r \ar@/_/[ddr]_r \\
& PX \times_X PX \ar[r]_(.6){q_2} \ar[d]^{q_1} & PX \ar[d]^{s} \\
& PX \ar[r]_{t} & X. }
The maps $s$ and $t$ are acyclic fibrations, and therefore their pullbacks $q_1$ and $q_2$ are acyclic fibrations as well; in particular, they are weak equivalences. Since $r$ is also a weak equivalence, the map $\alpha$ is a weak equivalence by 2-out-of-3. Therefore a suitable $\tau$ can be obtained as the lower filler of
\diag{ X \ar[r]^r \ar[d]_\alpha & PX \ar[d]^{(s, t)} \\
PX \times_X PX \ar[r]_(.55){(sq_1, tq_2)}  & X \times X.}
\end{enumerate}
This completes the proof.
\end{proof}

In a way $PX$ is the least homotopy equivalence relation on $X$.
\begin{lemm}{PXleastheqrel}
If $p = (p_1, p_2): R \to X \times X$ is a homotopy equivalence relation on $X$, then there is a map $h: PX \to R$ such that $p_1h = s$ and $p_2h = t$. More generally, any map $f: Y \to X$ gives rise to a morphism $h: PY \to R$ such that $p_1h = fs$ and $p_2h = ft$.
\end{lemm}
\begin{proof}
The square
\diag{ Y \ar[r]^{\rho f} \ar[d]_r & R \ar[d]^p \\
PY \ar[r]_(.4){(fs, ft)}  & X \times X}
has a lower filler, yielding the desired map $h$.
\end{proof}

\begin{theo}{homotopyiscongr}
The homotopy relation $\simeq$ defines a congruence relation on \ct{C}.
\end{theo}
\begin{proof}
We have already seen that if $P$ is a path object on $X$ and there is a suitable homotopy connecting $f$ and $g$ relative to $P$, then there is such a homotopy relative to any path object $Q$ for $X$. Therefore the statement that $\simeq$ defines an equivalence relation on each hom-set follows from \refprop{pathobjheq}.

For showing that $\simeq$ defines a congruence relation (i.e., that $f \simeq g$ and $k \simeq l$ imply $kf \simeq lg$), it suffices to prove that $f \simeq g$ implies $fk \simeq gk$ and $lf \simeq lg$; the former, however, is immediate, while the latter follows from \refcoro{onpathobj}.
\end{proof}

The previous theorem means that we can quotient \ct{C} by identifying homotopic maps and obtain a new category. The result is the \emph{homotopy category} of \ct{C} and will be denoted by ${\rm Ho}(\ct{C})$.

\begin{defi}{homeq}
A map $f: X \to Y$ is a \emph{homotopy equivalence} if it becomes an isomorphism in ${\rm Ho}(\ct{C})$ or, in other words, if there is a map $g: Y \to X$ (a \emph{homotopy inverse}) such that the composites $fg$ and $gf$ are homotopic to the identities on $Y$ and $X$, respectively. If such a homotopy equivalence $f: X \to Y$ exists, we say that $X$ and $Y$ are \emph{homotopy equivalent}.
\end{defi}

\begin{theo}{weisheq} Weak equivalences and homotopy equivalences coincide.
\end{theo}
\begin{proof} First note that any section of a weak equivalence $f: Y \to X$ is a homotopy inverse. The reason is that if $g: X \to Y$ is a section with $fg = 1$, then $g$ is a weak equivalence as well. Therefore we can find a homotopy $h: gf \simeq 1$ as a lower filler of
\diag{ X \ar[d]_g \ar[r]^{rg} & PY \ar[d]^{(s, t)} \\
Y \ar[r]_(.4){(gf, 1)} & Y \times Y. }
Since every acyclic fibration has a section, it now follows that acyclic fibrations are homotopy equivalences. But then \refcoro{makingweeasy} implies that every weak equivalence is a homotopy equivalence.

For the converse direction we also need to make a preliminary observation: if $f, g: A \to B$ are homotopic and $f$ is a weak equivalence, then so is $g$. To see this suppose that $f$ is a weak equivalence and there is a map $h: A \to PB$ such that $sh = f$ and $th = g$. Since $s$ and $t$ are weak equivalences, it follows from the first equality that $h$ is a weak equivalence and hence from the second equality that $g$ is a weak equivalence.

Now suppose $f: A \to B$ is a homotopy equivalence with homotopy inverse $g: B \to A$. Then in
\diag{ A \ar[r]^f & B \ar[r]^g & A \ar[r]^f & B }
both $gf$ and $fg$ are homotopic to the identity. Therefore both $gf$ and $fg$ are weak equivalences by the previous observation; but then we can use 2-out-of-6 to deduce that $f$ is a weak equivalence.
\end{proof}

\begin{rema}{on2outof6} Other authors who work in categorical frameworks similar to ours often call categories of fibrant objects ``saturated'' if they have the property that every homotopy equivalence is a weak equivalence. In our set-up this is derivable, so our path categories are always saturated in their sense. Note that in order to prove this we have made our first genuine use of the 2-out-of-6 axiom as opposed to the weaker 2-out-of-3 axiom: this is no coincidence, as relative to the 2-out-of-3 axiom the statement that every homotopy equivalence is a weak equivalence is equivalent to the 2-out-of-6 property (this observation is due to Cisinski; see \cite[p.~82--84]{radulescubanu09}). Using the theory that we will develop in the next subsection it will also not be hard to show that if a path category only satisfies 2-out-of-3 one can obtain a ``saturated'' path category from it by taking the same underlying category and the same fibrations, while enlarging the class of weak equivalences to include all homotopy equivalences. This means that restricting to saturated path categories is no real loss of generality; moreover, all the examples we are interested in are already saturated, including the syntactic category associated to type theory. For these reasons we have decided to restrict our attention to path categories that are saturated.
\end{rema}

\begin{coro}{weqsclosedunderretracts}
Weak equivalences are closed under retracts.
\end{coro}

\begin{coro}{univpropofho}
The quotient functor $\gamma: \ct{C} \to {\rm Ho}(\ct{C})$ is the universal solution to inverting the weak equivalences.
\end{coro}
\begin{proof}
We have just seen that this functor inverts the weak equivalences; conversely, any functor $\delta: \ct{C} \to  \ct{D}$ which sends weak equivalences to isomorphisms must identify homotopic maps, for if $PX$ is a path object with $r: X \to PX$ and $(s, t): PX \to X \times X$, then $\delta(s) = \delta(r)^{-1} = \delta(t)$.
\end{proof}

\subsection{Homotopy pullbacks.} Path categories need not have pullbacks; what they do have are homotopy pullbacks, a classical notion that we will now recall.

Given two arrows $f: A \to I$ and $g: B \to I$ one can take the pullback
\diag{ A\times_I^h B \ar[d]_{(p_1, p_2)} \ar[r] & PI \ar[d]^{(s, t)} \\
A \times B \ar[r]_{f \times g} & I \times I.}
This object $A \times_I^h B$ fits in a square
\diag{ A \times_I^h B \ar[r]^(.6){p_2} \ar[d]_{p_1} & B \ar[d]^g \\
A \ar[r]_f & I, }
which commutes up to homotopy.
\begin{defi}{hompbk}
Suppose
\diag{ C \ar[r]^{q_2} \ar[d]_{q_1} & B \ar[d]^g \\
A \ar[r]_f & I }
is a square which commutes up to homotopy. If there is a homotopy equivalence $h: C \to A \times_I^h B$ such that $q_i = p_i h$ for $i \in \{1, 2\}$, then the square above is called a \emph{homotopy pullback square} and $C$ is a \emph{homotopy pullback} of $f$ and $g$.
\end{defi}

\begin{rema}{comphompbks}
Clearly, a homotopy pullback is unique up to homotopy, but there are different ways of constructing it. For example, the homotopy pullback can be obtained by taking the fibrant replacement of either $f$ or $g$ (or both) and then taking the actual pullback. Using this one easily checks that the following well-known properties hold:
\begin{lemm}{propofhompbs}
\begin{enumerate}
\item[(i)] If
\diag{ D \ar[r] \ar[d]_g & B \ar[d]^f \\
C \ar[r] & A }
is a homotopy pullback and $f$ is a homotopy equivalence, then so is $g$.
\item[(ii)] If
\diag{ F \ar[d] \ar[r] & D \ar[r] \ar[d] & B \ar[d] \\
E \ar[r] & C \ar[r] & A }
commutes and the square on the right is a homotopy pullback, then the square on the left is a homotopy pullback if and only if the outer rectangle is a homotopy pullback.
\end{enumerate}
\end{lemm}
\end{rema}

\subsection{Connections and transport} One key fact about fibrations in path categories is that they have a path lifting property and allow for what the type-theorists call transport. The aim of this subsection is to show these facts, starting with the latter.

To formulate the notion of transport we need some additional terminology.
\begin{defi}{fibrewisehom}
Suppose $f, g : Y \to X$ are parallel arrows and $X$ comes with a fibration $p: X \to I$. If $pf = pg$, then we can compare $f: Y \to X$ and $g: Y \to X$ with respect to the path object $(s, t): P_I(X) \to X \times_I X$ of $X$ in $\ct{C}(I)$. Indeed, one calls $f$ and $g$ \emph{fibrewise homotopic} if there is a map $h: Y \to P_I(X)$ such that $sh = f$ and $th = g$; in that case one writes $f \simeq_I g$, or $h: f \simeq_I g$, if we wish to stress the homotopy. (If $pf = pg$ is a fibration, then this is just the homotopy relation in $\ct{C}(I)$; but, and this will be important later, this definition makes sense even when $pf = pg$ is not a fibration.)
\end{defi}

Recall from \refprop{factlemma} that any map $f: Y \to X$ can be factored as a weak equivalence $w_f: Y \to P_f$ followed by a fibration $p_f: P_f \to X$, where $P_f = Y \times_X PX$ is the pullback
\diag{ P_f \ar[r]^{p_2} \ar[d]_{p_1} & PX \ar[d]^s \\
Y \ar[r]_f & X,}
while $w_f = (1_Y, rf)$ and $p_f = tp_2$. If $f$ is a fibration, then we can regard both $Y$ and $P_f$ as objects in $\ct{C}(X)$ via $f$ and $p_f$, respectively, and $w_f$ as a morphism between them in $\ct{C}(X)$.

\begin{defi}{transport}
Let $f: Y \to X$ be a fibration. A \emph{transport structure} on $f$ is a morphism $\Gamma: P_f \to Y$ such that $f \Gamma = p_f$ and $\Gamma w_f \simeq_X 1_Y$.
\end{defi}

The idea behind transport is this: given an element $y \in Y$ and a path $\alpha: x \to x'$ in $X$ with $f(y) = x$, one can transport $y$ along $\alpha$ to obtain an element $y'$ with $f(y') = x'$; in addition, we demand that in case $\alpha$ is the identity path on $x$, then the element $y'$ should be connected to $y$ by a path which lies entirely in the fibre over $x$. In order to show that every fibration carries a transport structure, we need to strengthen \reflemm{lifting} to:
\begin{lemm}{exofweakfillers}
Suppose we are given a commutative square
\diag{ D \ar[r]^g \ar[d]_w & C \ar[d]^p \\
B \ar[r]_ f  & A}
in which $w$ is a weak equivalence and $p$ is a fibration. Then there is a map $l: B \to C$, unique up to homotopy, such that $p l = f$ and $l w \simeq g$.
\end{lemm}
\begin{proof}
We repeat the earlier proof: let $k: D \to B \times_A C$ be the map to the pullback with $p_1 k = w$ and $p_2 k = g$, and factor $k$ as $k = q i$ where $i$ is a weak equivalence and $q$ is a fibration. Then $p_1q$ is an acyclic fibration, so has a section $a$. So if we put $l := p_2qa$, then $pl = pp_2qa = fp_1qa = f$. But (the proof of) \reftheo{weisheq} implies that $ap_1q \simeq 1$, so that
\[ lw = p_2qaw = p_2qap_1k = p_2qap_1qi \simeq p_2qi = p_2k = g. \]
To see that $l$ is unique up to homotopy, note that, more generally, the fact that weak equivalences are homotopy equivalences implies that if $lw \simeq l'w$ and $w$ is a weak equivalence, then $l \simeq l'$.
\end{proof}

\begin{theo}{fibrwithtransport}
Every fibration $f: Y \to X$ carries a transport structure. Moreover, transport structures are unique up to fibrewise homotopy over $X$.
\end{theo}
\begin{proof}
If $f: Y \to X$ is a fibration then the commuting square
\diag{ Y \ar[r]^1 \ar[d]_{w_f} & Y \ar[d]^f \\
P_f \ar[r]_{p_f} & X}
does not only live in \ct{C}, but also in $\ct{C}(X)$. Applying the previous lemma to this square in $\ct{C}(X)$ gives one the desired transport structure.
\end{proof}

\begin{defi}{connection}
Let $f: Y \to X$ be a fibration. A \emph{connection} on $f$ consists of a path object $PY$ for $Y$, a fibration $Pf: PY \to PX$ commuting with the $r,s$ and $t$-maps on $PX$ and $PY$, together with a morphism $\nabla: P_f \to PY$ such that $Pf \circ \nabla = p_2$ and $s \nabla = p_1$.
\end{defi}

The idea behind a connection is this: given an element $y \in Y$ and a path $\alpha: x \to x'$ in $X$ with $f(y) = x$, the connection finds a path $\beta: y \to y'$ with $f(\beta) = \alpha$.

\begin{theo}{existenceconnections} Let $f: Y \to X$ be a fibration in a path category \ct{C} and assume that $PX$ is a path object on $X$ and $\Gamma: Y \times_X PX \to Y$ is a transport structure on $f$. Then we can construct a path object $PY$ on $Y$ and a fibration $Pf: PY \to PX$ with the following properties:
\begin{enumerate}
\item[(i)] $Pf$ commutes with the $r$, $s$ and $t$-maps on $PX$ and $PY$.
\item[(ii)] There exists a connection structure $\nabla: P_f \to PY$ with $t\nabla = \Gamma$.
\end{enumerate}
In particular, every fibration $f: Y \to X$ carries a connection structure.
\end{theo}
\begin{proof}
The proof will make essential use of the path object $P_X(Y)$ of $Y$ in $\ct{C}(X)$. We will write $\rho: Y \to P_X(Y)$ and $(\sigma, \tau): P_X(Y) \to Y \times_X Y$ for the factorisation of $Y \to Y \times_X Y$ as a weak equivalence followed by a fibration.

The idea is to construct $PY$ as $P_\Gamma$ in $\ct{C}(X)$, that is, as the following pullback:
\diag{ PY \ar[d]_{q_1} \ar[r]^(.45){q_2} & P_X(Y) \ar[d]^{\sigma} \\
P_f \ar[r]_{\Gamma} & Y. }
Since $\Gamma: P_f \to Y$ is a transport structure, there is a homotopy $h: \Gamma w_f \simeq 1$ in $\ct{C}(X)$. This allows us to factor the diagonal $Y \to Y \times Y$ as $(w_f, h): Y \to PY$ followed by $(p_1q_1, \tau q_2): PY \to Y \times Y$, so to prove that this defines a path object on $Y$ we need to show that the first map is a weak equivalence and the second a fibration. For the former, note that $q_1(w_f,h) = w_f$, where $w_f$ is a weak equivalence and $q_1$ is an acyclic fibration, as it is the pullback of $\sigma$. For the latter, note that $PY = P_\Gamma$ in $\ct{C}(X)$ can also be constructed as the pullback
\diag{ PY \ar[r]^{q_2} \ar[d]_{(q_1, \tau q_2)} & P_X(Y) \ar[d]^{(\sigma, \tau)} \\
P_f \times_X Y \ar[r]_{\Gamma \times_X Y} & Y \times_X Y,}
as in the proof of \refprop{factlemma}, so that $(q_1, \tau q_2)$ is a fibration. Moreover, \[ p_1 \times 1: P_f \times_X Y \to Y \times Y \] is a fibration as well, as it arises in the following pullback:
\diag{ P_f \times_X Y \ar[d]_{p_1 \times 1} \ar[r]^(.6){ p_2\pi_1} & PX \ar[d]^{(s, t)} \\
Y \times Y \ar[r]_{f \times f} & X \times X.}
So $(p_1 \times 1)(q_1, \tau q_2) = (p_1q_1,  \tau q_2)$ is a fibration, as desired. In addition, we have a map $Pf := p_2q_1: PY \to PX$, which is also fibration. We now check points (i) and (ii).

(i) We have to show that $Pf$ commutes with the maps $r, s, t$ on $PY$ and $PX$.
\begin{enumerate}
\item The $r$-map on $PY$ is $(w_f, h)$, and we have
\[ Pf \circ r_Y = p_2q_1(w_f, h) = p_2w_f = r_X \circ f. \]
\item The $s$-map on $PY$ is $p_1q_1$, and we have
\[ s_X \circ Pf = sp_2q_1 = fp_1q_1 = f \circ s_Y. \]
\item The $t$-map on $PY$ is $\tau q_2$, and we have
\[ t_X \circ Pf = tp_2q_1 = p_f q_1 = f\Gamma q_1 = f\sigma q_2 = f \tau q_2 = f \circ t_Y, \]
where we have used that $f \sigma = f \tau$ is the map exhibiting $P_X(Y)$ as an object of $\ct{C}(X)$.
\end{enumerate}

(ii): To construct the connection, we simply put $\nabla := (1, \rho \Gamma)$. Then
\[ s_Y \nabla = p_1q_1(1, \rho\Gamma) = p_1 \]
and
\[ Pf \circ \nabla = p_2q_1(1, \rho\Gamma) = p_2, \]
showing that $\nabla$ is indeed a connection. In addition, one has
\[ t_Y \nabla = \tau q_2(1,\rho \Gamma) = \tau \rho \Gamma = \Gamma, \]
as desired.
\end{proof}

\begin{rema}{invariantconnections} We have just shown that if $PX$ is any path object for $X$ and $f: Y \to X$ is any fibration, one can find a suitable path object $PY$ for $Y$ and a connection map $\nabla: P_f \to PY$ for that particular path object. From this it does not follow that if $P'Y$ is another path object for $Y$ then one can find a connection $\nabla': P_f \to P'Y$ as well: in that sense the notion of connection is not invariant. In view of \refcoro{onpathobj} we will have a map $\nabla': P_f \to P'Y$ with $s\nabla' = p_1$ and $ft\nabla' = tp_2$. We will occasionally meet such \emph{weak connections} as well, where the main point about such weak connections is:
\begin{coro}{aboutweakconnections}
Let $f: Y \to X$ be a fibration and $PY$ be an arbitrary path object for $Y$. Then there is a map $\nabla: P_f \to PY$ such that $s \nabla = p_1$ and $ft \nabla = tp_2$.
\end{coro}
\end{rema}

We conclude this subsection by noting the following consequence of \reftheo{existenceconnections}, which we will repeatedly use in what follows.
\begin{prop}{fillersuptohomotopy}
If a triangle
\diag{ & Y \ar[d]^p \\
Z \ar[r]_g \ar[ur]^f & X }
with a fibration $p$ on the right commutes up to a homotopy $h: pf \simeq g$, then we can also find a map $f': Z \to Y$, homotopic to $f$, such that for $f'$ the triangle commutes strictly, that is, $pf' = g$.
\end{prop}
\begin{proof}
Let $h: pf \simeq g$ be a homotopy and choose a path object $PY$ for $Y$, a fibration $Pp: PY \to PX$ and a connection structure $\nabla: P_p \to PY$ as in \reftheo{existenceconnections}. Put $h' = \nabla(f, h)$ and $f' := th'$. One may now calculate that
\[ p f' = pt\nabla(f, h) = t \circ Pp \circ \nabla \circ (f, h) = tp_2(f, h) = t h = g, \]
so the triangle commutes strictly for $f'$. Moreover, \[ sh' = s\nabla(f, h) = p_1(f, h) = f,  \]
so $h'$ is a homotopy between $f$ and $f'$.
\end{proof}

\subsection{Lifting properties} At various points (\reflemm{lifting} and \reflemm{exofweakfillers}) we have seen statements to the effect that weak equivalences have a weak lifting property with respect to the fibrations. \reflemm{lifting} said that if
\diag{ A \ar[r]^m \ar[d]_w & C \ar[d]^ p \\
B \ar[r]_n & D }
is a commutative square with a weak equivalence $w$ on the left and a fibration $p$ on the right, there is a map $l: B \to C$ such that $pl = n$. In \reflemm{exofweakfillers} we saw that $l$ could be chosen such that $lw \simeq m$. The aim of this subsection is to show that this can be strengthened even further: we can find a map $l$ such that $p l = n$ and $l w \simeq_D m$, where $\simeq_D$ is meant to indicate that $l w$ and $m$ are fibrewise homotopic over $D$ via $p: C \to D$. This seems to be the strongest lifting property which could reasonably be expected in our setting.

The proof that this stronger lifting property holds proceeds in several steps. It will be convenient to temporarily call the weak equivalences $w$ with the desired property \emph{good}. So a weak equivalence $w$ will be called good if in any square as the one above with a fibration $p$ on the right, there is a map $l: B \to C$ such that $pl = n$ and $lw \simeq_D m$.

\begin{lemm}{somefactsaboutgoodwes}
\begin{enumerate}
\item[(i)] Good weak equivalences are closed under composition.
\item[(ii)] In order to show that a weak equivalence is good we only need to consider the case where the map $n$ along the bottom of the square is the identity. In other words, a weak equivalence $w: A \to B$ is good whenever for any commuting triangle
\diag{ & C \ar[d]^p \\
A \ar[r]_w \ar[ur]^k & B }
in which $p$ is a fibration, the map $p$ has a section $j$ such that $j w \simeq_B k$.
\end{enumerate}
\end{lemm}
\begin{proof}
(i): Suppose $w_1$ and $w_2$ are good weak equivalences and there is a commuting square
\diag{ Z \ar[r]^m \ar[d]_{w_1} & C \ar[dd]^p \\
Y \ar[d]_{w_2} \\
X \ar[r]_n & D }
with a fibration $p$ on the right. From the fact that $w_1$ is good we get a map $t_1: Y \to C$ such that $pt_1 = nw_2$ and $t w_1 \simeq_D m$; then, from the fact that $w_2$ is good we get a map $t: X \to C$ such that $pt = n$ and $tw_2 \simeq_D t_1$. Then $tw_2w_1 \simeq_D t_1 w_1 \simeq_D m$, so $t$ is as desired.

(ii): Suppose that
\diag{ A \ar[r]^m \ar[d]_w & C \ar[d]^ p \\
B \ar[r]_n & D }
is a commuting square with a fibration $p$ on the right. Pulling back $p$ along $n$ we obtain a diagram of the form
\diag{ & E \ar[d]^{p'} \ar[r]^{n'} & C \ar[d]^p \\
A \ar[r]_w \ar[ur]^{(w, m)} & B \ar[r]_n & D}
in which $E = B \times_D C$ and $p'$ is a fibration. Note that \refprop{presbypb} implies that we can obtain a path object for $E$ in $\ct{C}(B)$ by pulling back the path object for $C$ in $\ct{C}(D)$ along $n: B \to D$, as in
\diag{ E \ar[r] \ar[d] & C \ar[d] \\
P_B(E) \ar[r] \ar[d] & P_D(C) \ar[d] \\
E \times_B E \ar[r] \ar[d] & C \times_D C \ar[d] \\
B \ar[r] & D,}
with all squares being pullbacks.

Now suppose that $w$ is a weak equivalence with the property formulated in the lemma. This means that $p'$ has a section $j$ such that $jw \simeq_B (w, m)$, as witnessed by some homotopy $A \to P_B(E)$. Composing this homotopy with the map $P_B(E) \to P_D(C)$ above we obtain a homotopy witnessing that $n'jw \simeq_D n'(w, m) = m$. So putting $l := n'j$, we have $pl = pn'j=np'j= n$ and $lw = n'jw \simeq_D m$, as desired.
\end{proof}

Our strategy for showing that any weak equivalence is good is to use the factorisation of any weak equivalence as a weak equivalence of the form $w_f: Y \to P_f$ followed by an acyclic fibration $p_f: P_f \to Y$. So once we have shown that any weak equivalence of the form $w_f: Y \to P_f$ is good and any acyclic fibration is good, we are done in view of part (i) of the previous lemma. We do the latter thing first.

\begin{prop}{onacyclicfibrations}
A fibration $f: B \to A$ is acyclic precisely when it has a section $g: A \to B$ with $gf \simeq_A 1_B$.
\end{prop}
\begin{proof}
If a fibration $f: B \to A$ has a section $g: A \to B$ with $gf \simeq_A1_B$, then $g$ is a homotopy inverse. So $f$ is a weak equivalence by \reftheo{weisheq}.

Conversely, if $f: B \to A$ is an acyclic fibration, then it has a section $g: A \to B$. From 2-out-of-3 for weak equivalences and $fg = 1_A$ it follows that $g$ is a weak equivalence. Therefore
\diag{ A \ar[d]_g \ar[r]^{rg} & P_A(B) \ar[d]^{(s, t)} \\
B \ar[r]_(.35){(gf, 1)} & B \times_A B }
is a commuting square with a weak equivalence on the left and a fibration on the right. A lower filler for this diagram is a fibrewise homotopy showing that $gf \simeq_A 1_B$.
\end{proof}

\begin{coro}{acyclicfibrgood}
Acyclic fibrations are good.
\end{coro}
\begin{proof}
We will use part (ii) of \reflemm{somefactsaboutgoodwes}. So suppose we are given a commuting triangle of the form
\diag{ & C \ar[d]^p \\
A \ar[r]_w \ar[ur]^k & B}
in which $p$ is a fibration and $w$ is an acyclic fibration. The previous proposition tells us that there is a map $a: B \to A$ such that $wa = 1_B$ and $aw \simeq_B 1_A$. But then $j := ka$ is a section of $p$ with $jw = kaw \simeq_B k$.
\end{proof}

To show that weak equivalences of the form $w_f: Y \to P_f$ are good, it will be useful to introduce a bit of terminology.
\begin{defi}{SDR}
A morphism $f: A \to B$ is a \emph{strong deformation retract} if there are a map $g: B \to A$, a path object $PB$ for $B$ and a homotopy $h: B \to PB$ such that
\[ gf = 1_A, \quad sh = fg, \quad th = 1, \quad \mbox{ and } \quad hf = rf. \]
\end{defi}

The reason is the following:
\begin{lemm}{usefullemma}
Strong deformation retracts are good weak equivalences.
\end{lemm}
\begin{proof}
Let $f: A \to B$ be a strong deformation retract and $g$ and $h$ be as in the definition. Strong deformation retracts are clearly homotopy equivalences, so they are weak equivalences as well. To show that they are also good, suppose that
\diag{ & C \ar[d]^q \\
A \ar[r]_f \ar[ur]^k & B }
is a commutative triangle in which $q$ is a fibration; we need to find a map $j$ such that $qj = 1_B$ and $jf \simeq_B k$. To this purpose, consider the following diagram:
\diag{ A \ar[r]^k \ar[d]_f & C \ar[r]^1 \ar[d]_{w_ q} & C \ar[d]^q \\
B \ar[r]_(.45){(kg, h)}   & P_q \ar[r]_{p_q}   \ar@{.>}[ur]_{\Gamma_q} & B. }
The left-hand square commutes, as
\[ (kg,h)f = (kgf, hf) = (k, rf) = (k, rqk) = (1,rq)k = w_q k.\]
Moreover, the arrows along the bottom compose to the identity on $B$, because $p_q(kg, h) = tp_2(kg, h) = th = 1$. This means that we can use the transport structure on $q$ with $q \Gamma_q = p_q$ and $\Gamma_q w_q \simeq_B 1$ to define $j$ as $\Gamma_q(kg, h)$.
\end{proof}

So it remains to show:
\begin{prop}{wfSDR}
For any morphism $f: Y \to X$ the weak equivalence $w_f: Y \to P_f$ is a strong deformation retract.
\end{prop}
\begin{proof} The main difficulty is to find a suitable path object for $P_f$. What we will do is take the following pullback:
\diag{ PX \times_X PY \times_X PX \ar[r] \ar[d]_{(\sigma, \tau)} & PY \ar[d]^{(s, t)} \\
P_f \times P_f \ar[r]_{p_1 \times p_1} & Y \times Y,}
where $\sigma$ and $\tau$ intuitively take a triple $(\alpha: f(y) \to x, \gamma: y \to y', \alpha': f(y') \to x')$ and produce $(y, \alpha)$ and $(y', \alpha')$, respectively. By construction $(\sigma, \tau)$ is a fibration. The reflexivity term $\rho: P_f \to PX \times_X PY \times_X PX$ is given by $(p_2, rp_1, p_2)$; in other words, by sending $(y, \alpha: f(y) \to x)$ to $(\alpha, r(y), \alpha)$. We have $\sigma\rho = \tau\rho = 1$, so to show that $\rho$ is a weak equivalence, it suffices to show this for $\sigma$; this map, however, is the pullback of the map on the left in
\diag{ PY \times_X PX \ar[rr] \ar[d] & & PX \ar[d]^s \\
PY \ar[r]_t \ar[d]_s & Y \ar[r]_f & X \\
Y }
along $p_1: P_f \to Y$ and hence an acyclic fibration. So we have described a suitable candidate for $PP_f$.

We know that $p_1 w_f = 1$, so to prove that $w_f$ is a strong deformation retract we need to find a homotopy $h: P_f \to PP_f$ such that $\sigma h = w_f p_1$, $\tau h =  1$ and $hw_f = \rho w_f$. We set
\[ h := (rfp_1, rp_1, p_2). \]
Then we can compute:
\[ \sigma h = \sigma(rfp_1, rp_1, p_2) = (srp_1, rfp_1) = (1, rf)p_1 = w_fp_1 \]
and
\[ \tau h = \tau(rfp_1, rp_1, p_2) = (trp_1, p_2) = (p_1, p_2) = 1. \]
In addition, the equations
\[ hw_f = (rfp_1, rp_1, p_2)(1, rf) = (rf, r, rf) \]
and
\[ \rho w_f = (p_2, rp_1, p_2)(1, rf) = (rf, r, rf) \]
hold, showing that $hw_f = \rho w_f$.
\end{proof}

We conclude that every weak equivalence is good, which we formulate more explicitly as follows.

\begin{theo}{allgood}
If
\diag{ A \ar[r]^m \ar[d]_f & C \ar[d]^ p \\
B \ar[r]_n & D }
is a commutative square with a weak equivalence $f$ on the left and a fibration $p$ on the right, then there is a filler $l: B \to C$ such that $n = pl$ and $l f \simeq_D m$. Moreover, such a filler is unique up to fibrewise homotopy over $D$.
\end{theo}
\begin{proof}
Any weak equivalence $f: A \to B$ can be factored as $w_f: A \to P_f$ followed by an acyclic fibration $p_f: P_f \to B$. The former is good by \reflemm{usefullemma} and \refprop{wfSDR}, while the latter is good by \refcoro{acyclicfibrgood}; so $f$ is good by part (i) of \reflemm{somefactsaboutgoodwes}.

It remains to show uniqueness of $l$: but if both $l$ and $l'$ are as desired, then $lf \simeq_D l'f$, so there is a fibrewise homotopy $h: A \to P_D(C)$ such that
\diag{ A \ar[r]^(.4){h} \ar[d]_f & P_D(C) \ar[d]^{(s, t)} \\
B \ar[r]_(.35){(l, l')} & D \times_C D }
commutes. A lower filler for this square is a fibrewise homotopy showing that $l \simeq_D l'$.
\end{proof}

We are now able to prove that the factorisations of maps as weak equivalences followed by fibrations are unique up to homotopy equivalence.

\begin{coro}{uniquenessoffact}
If a map $k: Y \to X$ can be written as $k = p a = q b$ where $a: Y \to A$ and $b: Y \to B$ are weak equivalences and $p: A \to X$ and $q: B \to X$ are fibrations, then $A$ and $B$ are homotopy equivalent; moreover, the homotopy equivalence $f: A \to B$ and homotopy inverse $g: B \to A$ can be chosen such that $qf = p$, $pg = q$, $fa \simeq_X b$, $gb \simeq_X a$, $gf \simeq_X 1$ and $fg \simeq_X 1$.
\end{coro}

This means in particular that any two path objects on an object $X$ are homotopy equivalent, where the homotopy equivalence and inverse can be chosen to behave nicely with respect to the $r,s,t$-maps, as in the statement of \refcoro{uniquenessoffact}.

\section{Homotopy exact completion}

\subsection{Exactness.} This section will be devoted to developing a notion of exact completion for path categories, generalising the exact completion of a category with finite limits, as in \cite{carboni95,carboniceliamagno82}. In fact, this homotopy exact completion, as we will call it, will coincide with the ordinary exact completion if we regard a category with finite limits as a path category in which every morphism is a fibration and only the isomorphisms are weak equivalences. Another feature of our account is that the category of setoids, studied in the type-theoretic literature (see, for example, \cite{hofmann97,bartheetal03}), is the homotopy exact completion of the syntactic category of type theory.

Initially, we will study this homotopy exact completion directly; in later stages we will use that it can also be obtained as the homotopy category of an intermediate path category (see \reftheo{hexcatoffibrantobj} below).

\begin{defi}{hex} Given a path category \ct{C} one may construct a new category as follows. Its objects are the homotopy equivalence relations, as defined in \refdefi{hequivalencerel}. A morphism from $(X, \rho: R \to X \times X)$ to $(Y, \sigma: S \to Y \times Y)$ is an equivalence class of morphisms $f: X \to Y$ for which there is a map $\varphi: R \to S$ making the square
\diag{ R \ar[d]_\rho \ar@{.>}[r]^\varphi & S \ar[d]^\sigma \\
X \times X \ar[r]_{f \times f} & Y \times Y }
commute; two such morphisms  $f: X \to Y$ and $g: X \to Y$ are identified in case there is a map $H: X \to S$ such that the triangle
\diag{ & S \ar[d]^\sigma \\
X \ar[r]_(.4){(f, g)} \ar@{.>}[ru]^H & Y \times Y }
commutes. This new category will be called the \emph{homotopy exact completion} of $\ct{C}$ and will be denoted by ${\rm Hex}(\ct{C})$.
\end{defi}

\begin{rema}{onnotofarrows} In what follows we will often denote objects of ${\rm Hex}(\ct{C})$ as pairs $(X, R)$, leaving the fibration $\rho: R \to X \times X$ implicit. If it is made explicit, then $\rho_1$ and $\rho_2$ denote the first and second projection $R \to X$, respectively. Also, we will not distinguish notationally between a morphism $f: X \to Y$ in \ct{C} which represents a morphism $(X, R) \to (Y, S)$ in ${\rm Hex}(\ct{C})$ and the morphism thus represented; we do not expect that these conventions will lead to confusion.
\end{rema}

\begin{rema}{onfillersinhex} In this definition we have asked for the existence of fillers making the diagrams commute strictly; however, in view of \refprop{fillersuptohomotopy}, it suffices if there are dotted arrows making the diagrams commute up to homotopy.
\end{rema}

Our first task is to show that ${\rm Hex}(\ct{C})$ is an exact category: that is, it is a regular category with well-behaved quotients of equivalence relations. For the convenience of the reader, we recall these notions here. (For more information we refer to part A of \cite{johnstone02a}, where exact categories are called effective regular.)
\begin{defi}{regularcat}
Let \ct{E} be a category. A map $f: B \to A$ in \ct{E} is a \emph{cover} if the only subobject of $A$ through which it factors is the maximal one given by the identity on $A$. A category \ct{C} is \emph{regular} if it has all finite limits, every morphism in \ct{C} factors as a cover followed by a mono and covers are stable under pullback.
\end{defi}

In a regular category a map is a cover iff it is a regular epi (meaning that it arises as a coequalizer) iff it is the coequalizer of its kernel pair (see \cite[Proposition A1.3.4]{johnstone02a}).

\begin{defi}{eqrel}
A subobject $R \subseteq X \times X$ is an \emph{equivalence relation} if for any object $P$ in \ct{E} the image of the induced map
\[ {\rm Hom}(P, R) \to {\rm Hom}(P, X) \times {\rm Hom}(P, X) \]
is an equivalence relation on ${\rm Hom}(P, X)$. In a regular category a \emph{quotient} of an equivalence relation is a cover $X \to Q$ such that
\diag{R \ar[r] \ar[d] & X \ar[d] \\ X \ar[r] & Q }
is a pullback. (Hence $\xymatrix{R \ar@<.5ex>[r] \ar@<-.5ex>[r] & X}$ is the kernel pair of $X \to Q$ and the latter is the coequalizer of the former.) A regular category \ct{E} is \emph{exact} if every equivalence relation in \ct{E} has a quotient.
\end{defi}

For showing that ${\rm Hex}(\ct{C})$ is an exact category, it will be convenient to introduce some notation. If $f: X \to Y$ is any map and $\sigma: S \to Y \times Y$ is a homotopy equivalence relation, then the pullback
\diag{ P \ar[r] \ar[d] & S \ar[d]^\sigma \\
X \times X \ar[r]_{f \times f} & Y \times Y }
is a homotopy equivalence relation on $X$, which will be denoted by $f^*\sigma: f^* S \to X \times X$. Moreover, if $R \to X \times X$ and $S \to X \times X$ are two homotopy equivalence relations, then $R \cap S$ is the homotopy equivalence relation obtained by taking the following pullback:
\diag{ R \cap S \ar[r] \ar[d] & S \ar[d] \\
R \ar[r] & X \times X.}

\begin{lemm}{hexcartesian}
The category ${\rm Hex}(\ct{C})$ has finite limits.
\end{lemm}
\begin{proof}
Since $(1,  P1) \cong (1,1)$ is the terminal object, it suffices to construct pullbacks. If $f: (Y, S) \to (X, R)$ and $g: (Z, T) \to (X, R)$ are two maps in ${\rm Hex}(\ct{C})$, their pullback $(W, Q)$ can be constructed by letting $W$ be the pullback
\diag{ W \ar[d] \ar[r] & R \ar[d] \\
Y \times Z \ar[r]_(.45){f \times g} & X \times X }
and by letting $Q$ be the homotopy equivalence relation on $W$ obtained by pulling back the homotopy equivalence relation $\pi_1^*S \cap \pi_2^* T$ on $Y \times Z$ along the map $W \to Y \times Z$.
\end{proof}

\begin{lemm}{monosinhex}
A morphism $f: (X, \rho: R \to X \times X) \to (Y, \sigma: S \to Y \times Y)$ is monic if and only if there is a morphism $h: f^*S \to R$ such that $\rho h = f^*\sigma$. Therefore every mono is isomorphic to one of the form $f: (X, f^*S) \to (Y, S)$.
\end{lemm}
\begin{proof} Use the description of pullbacks from the previous lemma and the fact that $m: A \to B$ is monic if and only if in the pullback
\diag{A \times_B A \ar[r]^(.6){p_2} \ar[d]_{p_1} & A \ar[d]^m \\
A \ar[r]_m & B }
we have $p_1 = p_2$.
\end{proof}

\begin{lemm}{hexregular}
The category ${\rm Hex}(\ct{C})$ is regular and the covers are those maps \[
f: (X, \rho: R \to X \times X) \to (Y, \sigma: S \to Y \times Y) \] for which there are maps $g: Y \to X$ and $h: Y \to S$ in \ct{C} such that $\sigma h = (1, fg)$.
\end{lemm}
\begin{proof}
Let us temporarily call maps $f$ as in the statement of the proposition \emph{nice epis}. Then the proposition follows as soon as we show:
\begin{enumerate}
\item Every map factors as a nice epi followed by a mono.
\item Nice epis are stable under isomorphism.
\item Nice epis are covers.
\item Nice epis are stable under pullback.
\end{enumerate}
This is all fairly easy: for example, if $f: (X, R) \to (Y, S)$ is any map, then it can be factored as
\diag{(X, R) \ar[r]^(.45)1 & (X, f^*S) \ar[r]^(.55)f & (Y, S),} where the first map is a nice epi and the second a mono. We leave it to the reader to check the other properties.
\end{proof}

We record the following corollary for future reference:
\begin{lemm}{epilemmainHex}
If $f: (X, R) \to (Y, S)$ is a cover in ${\rm Hex}(\ct{C})$, then $(X, f^*S) \cong (Y, S)$. Indeed, each cover in ${\rm Hex}(\ct{C})$ is isomorphic to one of the form $1: (X, R) \to (X, S)$. 
\end{lemm}
\begin{proof}
If $f: (X, R) \to (Y, S)$ is a cover, then $(Y, S)$ is isomorphic to the image of $f$, which, according to \reflemm{hexregular}, is precisely $(X, f^*S)$.
\end{proof}

\begin{lemm}{everymapafibr}
For every map $f: (X, R) \to (Y, S)$ there exists a factorisation $(X, R) \to (X', R') \to (Y, S)$ where the first is an isomorphism in ${\rm Hex}(\ct{C})$ and the second is represented by a fibration $X' \to Y$ in \ct{C}. In particular, every subobject of $(Y, S)$ has a representative via a map $f: (X, f^*S) \to (Y, S)$ where $f: X \to Y$ is a fibration.
\end{lemm}
\begin{proof}
The map $f$ can be factored as a homotopy equivalence $w_f: X \to X'$ followed by a fibration $p_f: X'\to Y$; this means that there is a map $i: X'\to X$ such that $w_fi \simeq 1$ and $iw_f \simeq 1$. One obtains $R'$ by pulling back $R$ along $i$. We leave the verification of the details to the reader.
\end{proof}

\begin{theo}{hexexact}
The category ${\rm Hex}(\ct{C})$ is exact.
\end{theo}
\begin{proof}
In view of the previous lemma it suffices to construct quotients of equivalence relations $f: (Y, f^*(R \times R)) \to (X \times X, R \times R)$ where $f: Y \to X \times X$ is a fibration and $(X, \rho: R
\to X \times X)$ is an object in ${\rm Hex}(\ct{C})$. But in this case one can take $(X, \tau: R \times_X Y \times_X R \to X \times X)$, where if we consider $R \times_X Y \times_X R$ heuristically as the set of triples $(r, y, r')$ with $\rho_2(r) = f_1(y)$ and $f_2(y) = \rho_1(r')$, then $\tau$ sends such a triple to $(\rho_1(r), \rho_2(r'))$. The map $1: (X, \rho) \to (X, \tau)$ is a cover and one easily verifies that its kernel pair is isomorphic to $(Y, f^*(R \times R))$.
\end{proof}

\begin{rema}{generalisedelements} The set-theoretic notation that we have used in the description of $\tau$ in the previous theorem can be justified in various ways, for example, by using generalised elements. In that case the description can be understood to say that $R \times_X Y \times_X R$ is an object such that maps into it from an object $I$ correspond bijectively to triples of maps $r: I \to P, y: I \to Y, r': I \to R$ with $\rho_2 r = f_1 y$ and $f_2 y = \rho_1 r'$. In addition, the existence of $\tau$ derives from the fact that the operation taking such triples $(r, y, r')$ to $(\rho_1r, \rho_2r')$ is a natural operation of the form
 \[ {\rm Hom}(I, R \times_X Y \times_X R) \to {\rm Hom}(I, X \times X). \]
But then it follows from the Yoneda Lemma that this operation must be given by postcomposition by some unique map $R \times_X Y \times_X R \to X \times X$. From now on we will increasingly rely on such heuristic set-theoretic descriptions; we trust that the reader can replace these descriptions by diagrammatic ones, if desired.
\end{rema}

\subsection{The homotopy exact completion as a homotopy category.} It turns out that one may also view the homotopy exact completion as a homotopy category. Indeed, it is often useful to regard the homotopy exact completion ${\rm Hex}(\ct{C})$ as the result of a two step procedure, where one first constructs out of a path category \ct{C} a new path category ${\rm Ex}(\ct{C})$, from which ${\rm Hex}(\ct{C})$ can then be obtained by taking the homotopy category.

The objects of ${\rm Ex}(\ct{C})$ are the same as those of ${\rm Hex}(\ct{C})$, that is, they are again homotopy equivalence relations, as defined in \refdefi{hequivalencerel}. However, a morphism from $(X, \rho: R \to X \times X)$ to $(Y, \sigma: S \to Y \times Y)$ in ${\rm Ex}(\ct{C})$ is a morphism $f: X \to Y$ for which there exists a map $\varphi: R \to S$ making the square
\diag{ R \ar[d]_\rho \ar@{.>}[r]^\varphi & S \ar[d]^\sigma \\
X \times X \ar[r]_{f \times f} & Y \times Y }
commute (we call such a map $\varphi$ a \emph{tracking}). For any two such arrows $f, g: X \to Y$ we will write $f \sim g$ if there is a map $H: X \to S$ such that $(f, g) = \sigma H: X \to Y \times Y$. This relation defines a congruence on ${\rm Ex}(\ct{C})$ and we will choose our fibrations and weak equivalences in such a way that this will become the homotopy relation on this path category, so that the homotopy category of ${\rm Ex}(\ct{C})$ will be equivalent to ${\rm Hex}(\ct{C})$.

A morphism $f$ as above is said to be a fibration in ${\rm Ex}(\ct{C})$ if:
\begin{enumerate}
\item[(1)] $f$ is a fibration in \ct{C}, and
\item[(2)] if $X \times_Y S$ is the pullback
\diag{ X \times_Y S \ar[r]^(.6){p_2} \ar[d]_{p_1} & S \ar[d]^{\sigma_1} \\
X \ar[r]_f & Y, }
there is a map $\nabla: X \times_Y S \to R$ in \ct{C} (``a weak connection structure'') such that $\rho_1 \nabla = p_1$ and $f \rho_2 \nabla = \sigma_2 p_2$.
\end{enumerate}
And $f$ will be a weak equivalence in ${\rm Ex}(\ct{C})$ if there is a map $g: (Y, S) \to (X, R)$ such that $fg \sim 1_Y$ and $gf \sim 1_X$.

\begin{lemm}{characyclicfibrinhex}
A fibration $f: (X, \rho: R \to X \times X) \to (Y, \sigma: S \to Y \times Y)$ in ${\rm Ex}(\ct{C})$ is acyclic if and only there is a map $a: Y \to X$ in \ct{C} such that $fa = 1_Y$ and $af \sim 1_X$. Indeed, such a map $a: Y \to X$ in \ct{C} will automatically be a map in ${\rm Ex}(\ct{C})$.
\end{lemm}
\begin{proof}
If $f$ is acyclic, there is a map $g: (Y, S) \to (X, R)$ such that $fg \sim 1_Y$ and $gf \sim 1_X$. The former gives one a map $H: Y \to S$ such that $\sigma H = (fg, 1)$. We put $a := \rho_2 \nabla(g, H)$. Then \[ fa = f\rho_2 \nabla(g, H) = \sigma_2 p_2(g, H) = \sigma_2 H = 1_Y \] and $\nabla(g, H)$ witnesses that $g \sim a$, so $af \sim gf \sim 1_X$.

Conversely, if $a: Y \to X$ is such that $fa = 1_Y$ and $af \sim 1_X$, then $a$ can be regarded as a map $(Y, S) \to (X, R)$. To show this, we use set-theoretic notation, as discussed in \refrema{generalisedelements}. If $s \in S$ connects $y_0$ and $y_1$, that is, if $\sigma_1(s) = y_0$ and $\sigma_2(s) = y_1$, then $t_0:= \nabla(a(y_0),s)$ connects $a(y_0)$ with some point $x$ such that $f(x) = y_1$. But then from the witness of $1_X \sim af$ we find a $t_1$ connecting $x$ and $a(f(x)) = a(y_1)$. So in order to obtain a tracking for $a$ we should send $s$ to the composition of $t_0$ and $t_1$, using the transitivity of $R$.
\end{proof}

\begin{theo}{hexcatoffibrantobj}
The category ${\rm Ex}(\ct{C})$ is a path category whose homotopy category is equivalent to ${\rm Hex}(\ct{C})$.
\end{theo}
\begin{proof}
We check the axioms.

(1) \emph{Fibrations are closed under composition.} If $f: (X, \rho: R \to X \times X) \to (Y, \sigma: S \to Y \times Y)$ and $g: (Y, \sigma: S \to Y \times Y) \to (Z, \tau: T \to Z \times Z)$ are fibrations with weak connections $\nabla^f: X \times_Y S \to R$ and $\nabla^g: Y \times_Z T \to S$, respectively, then $gf$ is a fibration with weak connection $\nabla^{gf}: X \times_Z T \to R$ defined by $\nabla^{gf}(x, t) := \nabla^f(x, \nabla^g(f(x), t))$.

(2) \emph{The pullback of a fibration along any map exists and is again a fibration.} If $f: (X, \rho: R \to X \times X) \to (Y, \sigma: S \to Y \times Y)$ is a fibration with weak connection $\nabla^f$ and $g: (Z, \tau: T \to Z \times Z) \to (Y, \sigma: S \to Y \times Y)$ is tracked by $\varphi$, then we can construct its pullback by taking $X \times_Y Z$ together with the homotopy equivalence relation $\pi_1^* R \cap \pi_2^*S$. The projection $X \times_Y Z \to Z$ has a weak connection structure: given a pair $(x_0, z_0)$ with $f(x_0) = g(z_0)$ and an element $t  \in T$ from $z_0$ to $z_1$, the element $s := \varphi(t) \in S$ connects $f(x_0) = g(z_0)$ and $g(z_1)$. So by the weak connection on $f$ one obtains an element $r \in R$ connecting $x_0$ to some $x_1$ with $f(x_1) = g(z_1)$. So $(r, t)$ connects $(x_0, z_0)$ to some point $(x_1, z_1)$ in $X \times_Y Z$ above $z_1$.

(3) \emph{The pullback of an acyclic fibrations along any map is again an acyclic fibration.} If in the situation as in (2) the map $f$ has a section $a$ with $af \sim 1$, then the projection $\pi_2: X \times_Y Z \to Z$ has a section $b$ defined by $b(z) = (a(g(z)),z)$. It is clear that $b \pi_2 \sim 1$.

(4) \emph{Weak equivalences satisfy 2-out-of-6.} This follows from the fact that $\sim$ is a congruence.

(5) \emph{Isomorphisms are acyclic fibrations and every acyclic fibration has a section.} Immediate from the previous lemma.

(6) \emph{The existence of path objects.} If $(X, \rho: R \to X \times X)$ is a homotopy equivalence relation with $e: X \to R$ witnessing reflexivity (so $\rho e = \Delta_X$), then we can factor the diagonal on $X$ as:
    \diag{ (X, R) \ar[r]^(.45)e & (R, \rho_1^* R) \ar[r]^(.4)\rho & (X \times X,\pi_1^* R \cap \pi_2^* R),} where $\rho_1: R \to X$ is an acyclic fibration left inverse to the first map. We leave the verifications to the reader.
    
(7) \emph{The category has a terminal object and every map to the terminal object is a fibration.} The terminal object is $(1, P1) \cong (1,1)$. The verification that the unique map $X \to 1$ is always a fibration $(X, R) \to (1, P1)$ in ${\rm Ex}(\ct{C})$ is trivial.

Note that it follows from the description of the path objects in (6) that the relation $\sim$ is precisely the homotopy relation in ${\rm Ex}(\ct{C})$: for if $f, g: (Y, \sigma: S \to Y \times Y) \to (X, \rho: R \to X \times X)$ are two parallel maps and $H: Y \to R$ witnesses that $f \sim g$, then $H$ is also a map $(Y, S) \to (R, \rho^*_1 R)$ which is tracked by any tracking of $f$. Therefore ${\rm Ex}(\ct{C})$ is a path category whose homotopy category is equivalent to ${\rm Hex}(\ct{C})$.
\end{proof}

\subsection{The embedding.} In the theory of exact completions of categories with finite limits the embedding of the original category into the exact completion plays an important r\^ole. For homotopy exact completions there is a similar functor
\[ i: \ct{C} \to {\rm Hex}(\ct{C}) \]
defined by sending $X$ to $(X, PX)$. In this subsection we will try to determine which properties from the theory of ordinary exact completions continue to hold and which ones seem to break down.

First of all, we should note that the functor $i$ is full, but not faithful: indeed, its image is equivalent to the homotopy category ${\rm Ho}(\ct{C})$.

In the ordinary theory of exact completion the functor $i$ preserves finite limits. This is not true here.
\begin{exam}{idoesnotpreservepullbacks}
The category of topological spaces has the structure of a path category if one takes the homotopy equivalences as its weak equivalences and the Hurewicz fibrations as the fibrations. The universal cover of the circle $p: \mathbb{R} \to S^1$ is a Hurewicz fibration which fits into a (homotopy) pullback as follows:
\diag{ \mathbb{Z} \ar[r] \ar[d] & \mathbb{R} \ar[d]^p \\
1 \ar[r] & S^1. }
The image of this square under $i$ is no longer a pullback, however, because $\mathbb{R}$ is contractible, while the discrete space $\mathbb{Z}$ is not.
\end{exam}
 
Instead one has:
\begin{prop}{isendspbkstoqpbks} \begin{enumerate}
\item[(1)] The functor $i$ preserves finite products. 
\item[(2)] If
\diag{ C \times_A B \ar[d]_q \ar[r]^(.6)g & B \ar[d]^p \\
C \ar[r]_f & A }
is a pullback square in \ct{C} in which $q$ and $p$ are fibrations, then the induced arrow $i(C \times_A B) \to iC \times_{iA} iB$ in ${\rm Hex}(\ct{C})$ is a cover. So the functor $i: \ct{C} \to {\rm Hex}(\ct{C})$ sends homotopy pullback squares to quasi-pullback squares.
\end{enumerate} \end{prop}
\begin{proof}
(1) From the fact that $P(X \times Y) \simeq PX \times PY$ and the description of finite limits in ${\rm Hex}(\ct{C})$ in \reflemm{hexcartesian} it follows that $i$ preserves products.

(2) By \reftheo{hexcatoffibrantobj} the pullback $iC \times_{iA} iB$ in ${\rm Hex}(\ct{C})$ is isomorphic to \[ (C \times_A B, q^* PC \times g^* PB), \] and hence the identity map from $(C \times_A B, P(C \times_A B))$ to this object is a cover by \reflemm{hexregular}.
\end{proof}

In the ordinary theory of exact completions the objects in the image of $i$ are, up to isomorphism, the projectives (an object $P$ in an exact category is projective if any cover $e: X \to P$ has a section). That does not seem to be the case here, but we do have the following result:

\begin{prop}{iiscovering}
The objects in the image of the functor $i: \ct{C} \to {\rm Hex}(\ct{C})$ are projective and each object in ${\rm Hex}(\ct{C})$ is covered by some object in the image of this functor.
\end{prop}
\begin{proof}
It follows immediately from \reflemm{PXleastheqrel} and the characterisation of covers in \reflemm{epilemmainHex} that objects of the form $i(X)$ are projective, while maps of the form $1: (X, PX) \to (X, R)$ are covers.
\end{proof}

\begin{prop}{conntoordintheory}
The category ${\rm Hex}(\ct{C})$ is the exact completion of ${\rm Ho}(\ct{C})$ as a weakly lex category, as defined in \cite{carbonivitale98}.
\end{prop}
\begin{proof}
This follows from \reftheo{hexexact} and \refprop{iiscovering} above and Theorem 16 in \cite{carbonivitale98}.
\end{proof}

The previous proposition means that the homotopy exact completion can be described in terms of pseudo-equivalence relations. This is occasionally useful, so we will spell this out here.

\begin{defi}{pseqrel}
Let $f = (f_1,f_2): R \to X \times X$ be an arbitrary map (not necessarily a fibration) in a path category \ct{C}. Then $f$ will be called a \emph{pseudo-equivalence relation}, if there are maps $\rho: X \to R, \sigma: R \to R$ and $\tau: P \to R$ witnessing reflexivity, symmetry and transitivity of this relation, where $P$ is the homotopy pullback of $f_1$ and $f_2$.
\end{defi}

An alternative definition of ${\rm Hex}(\ct{C})$ can now be given as follows: take as objects pairs $(X, R)$, where $R$ is a pseudo-equivalence relation on $X$. A morphism \[ (X, \rho: R \to X \times X) \to (Y, \sigma: S \to Y \times Y) \] is an equivalence class of morphisms $f: X \to Y$ for which there is an arrow $\varphi: R \to S$ making the square
\diag{ R \ar[d]_\rho \ar@{.>}[r]^\varphi & S \ar[d]^\sigma \\
X \times X \ar[r]_{f \times f} & Y \times Y }
commute up to homotopy; here two such arrows $f, g: X \to Y$ are equivalent if there is a map $h: X \to S$ such that $(f, g) \simeq \sigma h$.

\begin{prop}{thirddescrofhex}
The category just described is equivalent to ${\rm Hex}(\ct{C})$.
\end{prop}
\begin{proof}
This follows from \refprop{conntoordintheory}, but it is also quite straightforward to prove this directly. Indeed, any homotopy equivalence relation is also a pseudo-equivalence relation, so ${\rm Hex}(\ct{C})$ embeds into the category just described. Therefore it remains to check that any pseudo-equivalence relation $\rho: R \to X \times X$ is isomorphic to a homotopy equivalence relation in this category. But it can be shown quite easily using the lifting properties that if $\rho$ is factored as a homotopy equivalence $R \to \hat{R}$ followed by a fibration $\hat{\rho}: \hat{R} \to X \times X$, then $\hat{\rho}: \hat{R} \to X \times X$ is a homotopy equivalence relation.
\end{proof}

Another aspect of the theory of exact completions is that the subobject lattices of objects of the form $iX$ can be described concretely as a poset reflection.

\begin{prop}{subobjectsinHexC} Let $X$ be an object in a path category \ct{C}. The subobject lattice of $iX$ in ${\rm Hex}(\ct{C})$ is order isomorphic to the poset reflection of $\ct{C}(X)$.
\end{prop}
\begin{proof}
\reflemm{everymapafibr} tells us that every subobject of $iX$ in  ${\rm Hex}(\ct{C})$ has a representative given by a map $f: (Y, R) \to (X, PX)$ where $f$ is a fibration and $R = f^*PX$. If $h: (Z, S) \to (Y, R)$ is a map over $iX$ between two such representatives $g: (Z, S) \to (X, PX)$ and $f: (Y, R) \to (X, PX)$, then $fh \simeq g$. But then there is also a map $h': Z \to Y$ homotopic to $h$ such that $fh' = g$. Since $h$ and $h'$ are homotopic, $h'$ also has a tracking as a map $(Z, S) \to (Y, R)$ in ${\rm Hex}(\ct{C})$ and as such $h$ and $h'$ represent the same map. In fact, any map $h'$ such that $fh' = g$ will have tracking as a map $(Z, S) \to (Y, R)$ because we are assuming that $S = g^* PX$ and $R = f^* PX$. It follows that the subobject lattice of $iX$ in ${\rm Hex}(\ct{C})$ is the poset reflection of $\ct{C}(X)$, as claimed.
\end{proof}

Another aspect of the classical theory of exact completions is that exact completion and slicing commute. That fails for path categories; in fact, we only have the following.
\begin{prop}{onslicing}
Let \ct{C} be a path category and $X$ be an object in \ct{C}. Then ${\rm Hex}(\ct{C})/i(X)$ is a reflective subcategory of ${\rm Hex}(\ct{C}(X))$.
\end{prop}
\begin{proof}
Let us first take a closer look at ${\rm Hex}(\ct{C})/i(X)$. Objects in this category are morphisms $f: (Y, S) \to (X, PX)$ in ${\rm Hex}(\ct{C})$, that is, homotopy classes of arrows $f: Y \to X$ with a tracking $S \to PX$. Using the factorisation of arrows as homotopy equivalences followed by fibrations in ${\rm Ex}(\ct{C})$, we may assume that $f$ is an ${\rm Ex}(\ct{C})$-fibration. This means that we may assume that the objects in this category are pairs consisting of a fibration $f: Y \to X$ and a homotopy equivalence relation $\sigma: S \to Y \times Y$ for which there is a weak connection structure $\nabla: Y \times_X PX \to S$ as well as a map $S \to PX$ making
\diag{ S \ar@{.>}[r] \ar[d]_\sigma & PX \ar[d]^{(s, t)} \\
Y \times Y \ar[r]_{f \times f} & X \times X }
commute.

Furthermore, the morphisms in ${\rm Hex}(\ct{C})/i(X)$ are equivalence classes of arrows \[ \varphi: (g: Z \to X, R) \to (f: Y \to X, S) \] such that $f \circ \varphi \simeq g$ and for which a tracking $R \to S$ exists, while $\varphi$ and $\varphi'$ are equivalent in case there is map $H$ making
\diag{ & S \ar[d] \\
Z \ar@{.>}[ur]^H \ar[r]_(.4){(\varphi, \varphi')} & Y \times Y }
commute. Since we are assuming that $f$ is a fibration, it follows from \refprop{fillersuptohomotopy} that we may just as well assume that $\varphi$ satisfies $f \varphi = g$. If both $\varphi$ and $\varphi'$ are such representations, then they represent the same arrow in ${\rm Hex}(\ct{C})/i(X)$ if there is a dotted filler as in
\diag{ & T \ar[r] \ar[d] & S \ar[d] \\
Z \ar[r]_(.37){(\varphi, \varphi')} \ar@{.>}[ur] & Y \times_X Y \ar[r] & Y \times Y,}
where the square is a pullback.

This suggests the correct definition of the embedding $\rho: {\rm Hex}(\ct{C})/i(X) \to {\rm Hex}(\ct{C}(X))$. Note that objects in ${\rm Hex}(\ct{C}(X))$ consist of pairs $(f: Y \to X, T \to Y \times_X Y)$, where $f$ is a fibration and $T \to Y \times_X Y$ is a homotopy equivalence relation in $\ct{C}(X)$. So we can define a functor $\rho: {\rm Hex}(\ct{C})/i(X) \to {\rm Hex}(\ct{C}(X))$ by sending $(f: Y \to X, S)$ to $f$  together with the homotopy equivalence relation in $\ct{C}(X)$ obtained as the pullback
\diag{ T \ar[r] \ar[d] & S \ar[d] \\
Y \times_X Y \ar[r] & Y \times Y. }

This functor $\rho$ has a left adjoint $\lambda: {\rm Hex}(\ct{C}(X)) \to {\rm Hex}(\ct{C})/i(X)$. The quickest way to define it is to use the factorisation in \ct{C}: starting from a pair $(f: Y \to X, T \to Y \times_X Y)$ we can factor the composition of $T \to Y \times_X Y$ with the inclusion $Y \times_X Y \to Y \times Y$ as a homotopy equivalence followed by a fibration:
\diag{ T \ar[r]^{\sim} \ar[d] & S \ar[d] \\
Y \times_X Y \ar[r] & Y \times Y.}
Using the lifting properties one can now show that $S \to Y \times Y$ is a homotopy equivalence relation and that $\lambda$ defines a left adjoint to $\rho$.

To complete the proof we have to show that $\lambda \rho \cong 1$. So suppose we are given a fibration $f: Y \to X$ and a homotopy equivalence relation $\sigma: S \to Y \times Y$ for which there are a weak connection $\nabla: Y \times_X PX \to S$ as well as a tracking $S \to PX$. Construct the following four pullbacks:
\diag{ T \ar[r] \ar[d] & S^* \ar[r] \ar[d] & S \ar[d] \\
Y \times_X Y \ar[r] \ar[d] & Y \times_X PX \times_X Y \ar[r] \ar[d] & Y \times Y \ar[d] \\
X \ar[r]_r & PX \ar[r]_{(s, t)} & X \times X.}
Note that all four arrows in the lower right-hand square are fibrations; since $S \to Y \times Y$ is a fibration, the maps $S^* \to Y \times Y$ and $S^* \to PX$ are fibrations as well. From the latter it follows that $T \to S^*$ is a weak equivalence since $r: X \to PX$ is. Therefore applying $\rho$ to $(f, S)$ yields $T \to Y \times_X Y$ and the result of applying $\lambda$ to that is $S^* \to Y \times Y$. Therefore it remains to construct a suitable map $S \to S^*$ over $Y \times Y$: but the existence of such a map follows from the universal property of $S^*$ and the existence of a tracking $S \to PX$.
\end{proof}

\begin{rema}{onlyreflective}
The adjunction $\lambda \ladj \rho$ in the proof above is not an equivalence: indeed, again consider the category of topological spaces, and take for $f: Y \to X$ the universal cover $\mathbb{R} \to S^1$ of the circle and let $\Delta: Y \to Y \times_X Y$ be the diagonal. Then $\lambda(f, \Delta) \cong (\mathbb{R}, P\mathbb{R}) \cong 1$ and $\rho\lambda(f, \Delta) \cong 1$, but $(f, \Delta) \not\cong 1$. In the same way one can show that $\lambda$ does not preserve finite products (if it would our treatment of $\Pi$-types below could have been simplified considerably). For $\lambda(f, \Delta) \times \lambda(f, \Delta) \cong 1$, while $(f, \Delta) \times (f, \Delta)$ is $Y \times_X Y$ with the diagonal, so \[ \lambda((f, \Delta) \times (f, \Delta)) = (Y \times_X Y, P(Y \times_X Y)), \] which is isomorphic to $\mathbb{Z}$ with the discrete topology.
\end{rema}

In the remainder of this section we will try to characterise the image of $\rho$ in ${\rm Hex}(\ct{C}(X))$. In order to do this, we introduce the following notion.

For the moment, fix a fibration $f: Y \to X$ and a homotopy equivalence relation $\tau: T \to Y \times_X Y$ in $\ct{C}(X)$; so, in effect, we are fixing an object in ${\rm Hex}(\ct{C}(X))$.
\begin{defi}{Ttransport} A \emph{transport structure relative to $T$}, or a \emph{T-transport}, is a map $\Gamma: Y \times_X PX \to Y$ such that:
\begin{enumerate}
\item $f\Gamma = t p_2$, and
\item there is a map $L: Y \to T$ such that $\tau L = (1, \Gamma(1, rf))$.
\end{enumerate}
\end{defi}

\begin{prop}{Ttranspunique}
$T$-transports exist and are unique up to $T$-equivalence; more precisely, if $\Gamma$ and $\Gamma'$ are two $T$-transports, there will be a map $H: Y \times_X PX \to T$ such that $\tau H = (\Gamma, \Gamma')$.
\end{prop}
\begin{proof}
For $T = P_X(Y)$ a $T$-transport structure is the same thing as an ordinary transport structure. Because there will always be a map $P_X(Y) \to T$ over $Y \times_X Y$, every ordinary transport structure is also a transport structure relative to $T$. In particular, transport structures relative to $T$ exist since ordinary ones do.

To show essential uniqueness, let $\Gamma$ and $\Gamma'$ be two $T$-transports. Then $\Gamma(1, rf)$ and $\Gamma'(1, rf)$ will be $T$-equivalent, as they are both $T$-equivalent to the identity on $Y$. This means that there is a map $K$ making the square
\diag{ Y \ar[d]_{(1, rf)} \ar[r]^K & T \ar[d]^\tau \\
Y \times_X PX \ar[r]_{(\Gamma, \Gamma')} & Y \times_X Y }
commute. But since $\tau$ is a fibration and $(1, rf)$ is a weak equivalence, we get the desired map $H$ from the usual lifting properties.
\end{proof}

\begin{prop}{TtransppresTequivalence}
$T$-transports preserve $T$-equivalence. More precisely, if $\Gamma$ is a $T$-transport, there will be a map $H: T \times_X PX \to T$ such that \[ \tau_1 H = \Gamma(\tau_1 p_1, p_2) \quad \mbox{ and } \quad \tau_2 H = \Gamma(\tau_2 p_1, p_2). \]
\end{prop}
\begin{proof}
If $\Gamma$ is a $T$-transport, then
\[ \Gamma(1, rf)\tau_1 \simeq_T \tau_1 \simeq_T \tau_2 \simeq_T \Gamma(1, rf)\tau_2: T \to Y. \]
Therefore there is a map $K$ making the diagram
\diag{ T \ar[rrrrr]^K \ar[d]_{(1, rf\tau_1)} & & & & & T \ar[d]^\tau \\
T \times_X PX \ar[rr]_(.4){\tau \times_X 1} & & Y \times_X Y \times_X PX \ar[rrr]_(.55){(\Gamma(p_1, p_3), \Gamma(p_2, p_3))} & & & Y \times_X Y }
commute, and $H$ is obtained as a lower filler of this diagram.
\end{proof}

\begin{defi}{stableelement} Let $(f, T)$ be an element of ${\rm Hex}(\ct{C}(X))$, so $f: Y \to X$ is a fibration and $T \to Y \times_X Y$ is a homotopy equivalence relation in $\ct{C}(X)$. We call such an object \emph{stable} if the action of loops in $X$ on the fibres of $f$ by the (essentially unique) $T$-transport $\Gamma: Y \times_X PX \to Y$ is $T$-trivial: so if $f(y) = x$ and $\alpha$ is a loop at $x$, then $\Gamma_\alpha(y) \simeq_T y$.
\end{defi}

\begin{theo}{equivwithstableelements}
Let \ct{C} be a path category and $X$ be an object in \ct{C}. Then ${\rm Hex}(\ct{C})/i(X)$ is equivalent to the full subcategory of ${\rm Hex}(\ct{C}(X))$ consisting of the stable objects.
\end{theo}
\begin{proof} In \refprop{onslicing} we have shown that ${\rm Hex}(\ct{C})/i(X)$ is a reflective subcategory of ${\rm Hex}(\ct{C}(X))$ and we gave explicit constructions of both the embedding $\rho$ and reflector $\lambda$ in the proof of that proposition. Note that objects in the image of $\rho$ are always stable: for suppose $T$ is the restriction to $Y \times_X Y$ of some homotopy equivalence relation $\sigma: S \to Y \times Y$ over $(X, PX)$. We may assume that $f: (Y, S) \to (X, PX)$ is an ${\rm Ex}(\ct{C})$-fibration, so that there is a weak connection structure $\nabla: Y \times_X PX \to S$. From this we obtain a $T$-transport $\Gamma$ given by $\Gamma = \sigma_2 \nabla$. If $f(y) = x$ and $\alpha$ is a loop at $x$, the weak connection $\nabla$ tells us that $\Gamma_\alpha(y) \simeq_S y$; but then also $\Gamma_\alpha(y) \simeq_T y$, by definition of $T$.

Conversely, let $(f: Y \to X, \tau: T \to Y \times_X Y)$ be an element of ${\rm Hex}(\ct{C}(X))$, and let $\Gamma: Y \times_X PX \to Y$ be the essentially unique $T$-transport. Compute the following pullbacks:
\diag{
T^* \ar[d]_{\tau^*}  \ar[r] &  S \ar[d] \ar[r] &  T \ar[d]^{\tau}  \\
Y \times_X Y \ar[r] \ar[d] &  Y \times_X PX \times_X Y \ar[dr] \ar[r]_(.6){\Gamma \times_X 1} \ar[d] & Y \times_X Y  \\
 X \ar[r]_r & PX \ar[dr]_{(s,t)} &  Y \times Y \ar[d]^{f \times f}  \\
& &   X \times X.}
In terms of generalised elements, \[ T^* = \{ \, (y, t, y\rq{}) \, : \, \tau(t) = (\Gamma(y,rfy), y\rq{}) \, \}.\] Since $\Gamma$ is a $T$-transport it follows that $\Gamma(y,rfy) \simeq_T y$ and therefore $(Y, T)$ and $(Y, T^*)$ are isomorphic in ${\rm Hex}(\ct{C}(X))$. So in order to compute $\lambda(Y, T)$ we might just as well compute $\lambda(Y, T^*)$, and because $T^* \to S$, as a pullback of $r: X \to PX$ along a fibration, is a weak equivalence, we see that $\lambda(Y, T^*)$ is $(Y, S \to Y \times Y)$. Therefore $\rho\lambda(Y, T)$ is the element in ${\rm Hex}(\ct{C}(X))$ consisting of $f: Y \to X$ together with the following homotopy equivalence relation in $\ct{C}(X)$: $y_1$ and $y_2$ over the same $x$ are related if there is a loop $\alpha$ on $x$ such that $\Gamma_\alpha(y_1) \simeq_T y_2$. But if $(Y, T)$ is stable, this is equivalent to $y_1 \simeq_T y_2$; so in this case $\rho\lambda(Y, T) \cong (Y, T)$.
\end{proof}

This theorem gives us a useful way of thinking about the slice category ${\rm Hex}(\ct{C})/i(X)$: especially when we have to deal with function spaces in ${\rm Hex}(\ct{C})/i(X)$, it is more  convenient to think about the stable elements in ${\rm Hex}(\ct{C}(X))$.

\section{Sums and the natural numbers object}

This section will be devoted to a study of the homotopy initial objects and homotopy sums in a path category, as well as a homotopy-theoretic version of the natural numbers object. We will define these as objects that become ordinary sums or the usual natural numbers object in the homotopy category. 

\subsection{Definition.} We will start our discussion with the homotopy versions of the initial object and binary coproducts.

\begin{defi}{hinitialobj}
An object $0$ is \emph{homotopy initial} if for any object $A$ there is a map $f: 0 \to A$ and any two such maps are homotopic. A \emph{homotopy sum} or \emph{homotopy coproduct} of two objects $A$ and $B$ is an object $A + B$ together with two maps $i_A: A \to A + B$ and $i_B: B \to A + B$ such that for any pair of maps $f: A \to X$ and $g: B \to X$ there is a map $h: A + B \to X$, unique up to homotopy, such that $hi_A \simeq f$ and $hi_B \simeq g$.
\end{defi}

This is not quite what the type theorist would expect: the type-theoretic axiom for the initial object, for example, says that any fibration $A \to 0$ has a section. However, this condition turns out to be equivalent.
\begin{prop}{othercharhinitialobj}
In a path category an object $0$ is homotopy initial if and only if any fibration $f: A \to 0$ has a section.
\end{prop}
\begin{proof}
Suppose we are given a fibration $f: A \to 0$. If $0$ is homotopy initial, then there is a map $g: 0 \to A$ with $fg \simeq 1$. So by \refprop{fillersuptohomotopy} there is a map $g': 0 \to A$ such that $fg' = 1$.

Conversely, suppose $0$ is such that any fibration $A \to 0$ has a section. For any object $B$ the second projection $\pi_2: B \times 0 \to 0$ is a fibration, so there is a map $a: 0 \to B \times 0$ such that $\pi_2 a = 1$; but then $f = \pi_1 a$ is a map $0 \to B$. In addition, if $g: 0 \to B$ is another map, then we can take the pullback
\diag{ Q \ar[d] \ar[r] & PB \ar[d]^{(s, t)} \\
0 \ar[r]_(.4){(f, g)} & B \times B}
giving rise to a fibration $Q \to 0$. This map has a section, and composing this section with the map $Q \to PB$ gives rise to a homotopy between $f$ and $g$.
\end{proof}

In the same way one has:

\begin{prop}{othercharhomsum}
An object $A + B$ together with maps $i_A: A \to A + B$ and $i_B: B \to A + B$ is the homotopy sum of $A$ and $B$ if and only if for any fibration $p: C \to A + B$ and any pair of maps $a: A \to C$ and $b: B \to C$ such that $pa = i_A$ and $pb = i_B$, there is a map $\sigma: A + B \to C$ such that $p\sigma = 1, \sigma i_A \simeq a$ and $\sigma i_B \simeq b$.
\end{prop}
\begin{proof}
$\Rightarrow$: Suppose we are given a fibration $p: C \to A + B$ together with maps $a: A \to C$ and $b: B \to C$ such that $pa = i_A$ and $pb = i_B$. We know that there is a map $h: A + B \to C$ such that $h i_A \simeq a$ and $h i_B \simeq b$. In addition, we must have $p h \simeq 1$, so by \refprop{fillersuptohomotopy} there is a map $\sigma: A + B \to C$ such that $p \sigma = 1$ and $\sigma \simeq h$; hence $\sigma i_A \simeq hi_A \simeq a$ and $\sigma i_B \simeq hi_B \simeq b$.

$\Leftarrow$: Let $f: A \to X$ and $g: B \to X$ be two maps. We want to show that there is a map $h: A + B \to X$, unique up to homotopy, such that $hi_A \simeq f$ and $hi_B \simeq g$. Put $C = X \times (A + B)$ and consider the projection $\pi_2: C \to A + B$ together with the maps $(f, i_A): A \to C$ and $(g, i_B): B \to C$. By assumption, there is a map $\sigma: A + B \to C$ such that $\pi_2 \sigma = 1$, $\sigma i_A \simeq (f, i_A)$, $\sigma i_B \simeq (g, i_B)$. So if we put $h = \pi_1 \sigma: A + B \to X$, then $h i_A \simeq f$ and $h i_B \simeq g$, as desired. If $h': A + B \to X$ satisfies the same equations, then we can take the following pullback:
\diag{ Q \ar[r]^{p_2} \ar[d]_{p_1} & PX \ar[d]^{(s, t)} \\
A + B \ar[r]_(.45){(h, h')} & X \times X}
giving rise to a fibration $Q \to A + B$. In addition, since $(h, h') i_A \simeq (f, f) = (s,t)rf$, there is a map $u: A \to PX$ such that $(s, t)u = (h, h')i_A$ and a map $k: A \to Q$ such that $p_1k = i_A$; similarly, there is a map $l: B \to Q$ such that $p_1l = i_B$. So $p_1$ has a section and composing this section with $p_2$ yields the desired homotopy between $h$ and $h'$.
\end{proof}

\begin{prop}{propofsumsinfibcategories}
Suppose \ct{C} is a path category with homotopy sums.
\begin{enumerate}
\item[(i)] If 0 is homotopy initial, then $0 + X \simeq X$ for any object $X$.
\item[(ii)] $f + g: X + Y \to A + B$ will be a homotopy equivalence if both $f: X \to A$ and $g: Y \to B$ are.
\item[(iii)] $P(A + B) \simeq PA + PB$.
\end{enumerate}
\end{prop}
\begin{proof} Parts (i) and (ii) are immediate consequences of the fact that homotopy equivalences are precisely those maps which become isomorphisms in the homotopy category, while homotopy initial objects become initial objects and homotopy sums become ordinary sums in the homotopy category.

(iii): It follows from (ii) that the canonical map $A + B \to PA + PB$ is a weak equivalence.  So the lifting properties give us a map $PA + PB \to P(A + B)$ making the top triangle in
\diag{ A + B \ar[rr] \ar[d] & & P(A + B) \ar[d] \\
PA + PB \ar[r] \ar@{.>}[urr] & A \times A + B \times B \ar[r] & (A + B) \times (A + B) }
commute up to homotopy. Since the map along the top is a homotopy equivalence, so is $PA + PB \to P(A + B)$.
\end{proof}

\subsection{Homotopy extensive path categories.}

For later purposes we do not only need homotopy sums to exist, but they should also have properties like disjointness and stability, as ordinary categorical sums have in an extensive category (see \cite{carbonietal93}). So we need a suitable notion of extensivity for path categories.

\begin{defi}{homextensive} Suppose \ct{C} is a path category.
\begin{enumerate}
\item A homotopy sum $A+B$ in \ct{C} is \emph{stable}, if for any diagram of the form
\diag{ C \ar[r] \ar[d] & X \ar[d] & D \ar[d] \ar[l] \\
A \ar[r] & A + B & B, \ar[l] }
the top row is a homotopy coproduct whenever both squares are homotopy pullbacks.
\item A homotopy sum $A + B$ is \emph{disjoint} if the square
\diag{ 0 \ar[r] \ar[d] & B \ar[d] \\
A \ar[r] & A + B }
is a homotopy pullback.
\item If \ct{C} has a homotopy inital object and homotopy sums which are both stable and disjoint, then \ct{C} will be called \emph{homotopy extensive}.
\end{enumerate}
\end{defi}

\begin{prop}{consequencesofstabilityofsums}
Let \ct{C} be a path category with stable homotopy sums and a homotopy initial object.
\begin{enumerate}
\item[(i)] The distributive law $X \times (A + B) \simeq X \times A + X \times B$ holds.
\item[(ii)] The homotopy initial object $0$ is strict: any map $X \to 0$ is a homotopy equivalence.
\item[(iii)] The functor $\ct{C}(A + B) \to \ct{C}(A) \times \ct{C}(B)$ is homotopy conservative (i.e., detects homotopy equivalences).
\end{enumerate}
\end{prop}
\begin{proof} Property (i) is a special case of stability, as applied to the following diagram:
\diag{ X \times A \ar[r] \ar[d] & X \times (A + B) \ar[d] & X \times B \ar[d] \ar[l] \\ A \ar[r] & A + B & B. \ar[l] }

To prove (ii), note that given any arrow $f: X \to 0$ the diagram
\diag{ X \ar[d]_f \ar[r]^1 & X \ar[d]^f & X \ar[d]^f \ar[l]_1 \\
0 \ar[r]_1 & 0 & 0 \ar[l]^1 }
consists of two (homotopy) pullbacks. So the top row is homotopy coproduct diagram by stability and therefore any two parallel arrows with domain $X$ are homotopic. This, in combination with the existence of a map $f: X \to 0$, implies that $X$ is a homotopy initial object and $f$ is a homotopy equivalence.

To prove (iii), suppose $f: Y \to X$ is a map in $\ct{C}(A + B)$ and let $f_A: Y_A \to X_A$ and $f_B: Y_B \to X_B$ be the pullbacks of $f$ along $A \to A + B$ and $B \to A + B$, respectively. If both $f_A$ and $f_B$ are homotopy equivalences, then so is $f_A + f_B: Y_A + Y_B \to X_A + X_B$ by \refprop{propofsumsinfibcategories}.(ii). But if the sums in \ct{C} are stable, then $Y_A + Y_B \simeq Y$ and $X_A + X_B \simeq X$, so $f$ is a homotopy equivalence, as desired.
\end{proof}

\begin{prop}{altdefhomext}
Suppose \ct{C} is a path category which has a homotopy initial object and homotopy sums. Then \ct{C} is homotopy extensive if and only if the following two conditions are satisfied:
\begin{enumerate}
\item[(i)]  If $C \to A$ and $D \to B$ are two maps, then
\diag{ C \ar[r] \ar[d] & C + D \ar[d] & D \ar[d] \ar[l] \\
A \ar[r] & A + B & B, \ar[l] }
consists of two homotopy pullbacks.
\item[(ii)] The functor $\ct{C}(A + B) \to \ct{C}(A) \times \ct{C}(B)$ is homotopy conservative.
\end{enumerate}
\end{prop}
\begin{proof} $\Rightarrow$: In view of \refprop{consequencesofstabilityofsums}.(iii) it remains to show that (i) holds in all homotopy extensive path categories. To this  purpose consider a homotopy pullback of the form
\diag{ C' \ar[r] \ar[d] & C + D \ar[d] \\
A \ar[r] & A + B.}
We would like to show that $C' \simeq C$, and since we have already shown that the functor $\ct{C}(C + D) \to \ct{C}(C) \times \ct{C}(D)$ is homotopy conservative, it suffices to prove that the following two squares are homotopy pullbacks:
\begin{equation} \label{twosquares}
\begin{array}{cc}
\xymatrix{ C \ar[d] \ar[r] & C \ar[d] \\
C' \ar[r] & C + D} &
\xymatrix{ 0 \ar[d] \ar[r] & D \ar[d] \\
C' \ar[r] & C + D}
\end{array}
\end{equation}
By pasting of homotopy pullbacks, the second square is a homotopy pullback if and only if
\diag{ 0 \ar[r] \ar[d] & D \ar[d] \\
A \ar[r] & A + B  }
is. But the latter square can be decomposed as
\diag{ 0 \ar[r] \ar[d] & D \ar[d] \\
0 \ar[d] \ar[r] & B \ar[d] \\
A \ar[r] & A + B. }
Here the bottom square is a homotopy pullback by the disjointness of the homotopy sums and the top square is a homotopy pullback by \refprop{consequencesofstabilityofsums}.(ii). We conclude that the second square in (\ref{twosquares}) is a homotopy pullback.

In the same way one can show that
\diag{0 \ar[d] \ar[r] & C \ar[d] \\
D' \ar[r] & C + D}
is a homotopy pullback, where $D'$ is the homotopy pullback in
\diag{ D' \ar[r] \ar[d] & C + D \ar[d] \\
B \ar[r] & A + B.}
Now consider
\diag{ C'' \ar[d] \ar[r] & C \ar[d] & 0 \ar[l] \ar[d] \\
C' \ar[r] \ar[d] & C + D \ar[d] & D' \ar[l] \ar[d] \\
A \ar[r] & A + B & B \ar[l]}
in which all squares are homotopy pullbacks. By stability of sums we have that $C \simeq C'' + 0 \simeq C''$. This shows that also the first square in (\ref{twosquares}) is a homotopy pullback.

$\Leftarrow$: Suppose (i) and (ii) are satisfied. To show that the homotopy sums are stable, suppose that
\diag{ C \ar[r] \ar[d] & X \ar[d] & D \ar[d] \ar[l] \\
A \ar[r] & A + B & B, \ar[l] }
consists of two homotopy pullbacks. We have to show $C + D \simeq X$. Without loss of generality we may assume that both $X \to A + B$ and $C + D \to A + B$ are fibrations. Therefore it suffices to prove that $C + D$ and $X$ are homotopy equivalent after pulling back along $A \to A + B$ and $B \to A + B$. But for both $C + D$ and $X$ the results are homotopy equivalent to $C$ and $D$, respectively, so $C + D \simeq X$.

To see that homotopy sums are disjoint, note that (i) implies that
\diag{ 0 \ar[r] \ar[d] & 0 + B \ar[d]  \\
A \ar[r] & A + B }
is a homotopy pullback.
\end{proof}

\begin{prop}{propofhomextcats}
Let \ct{C} be a homotopy extensive path category. If the following squares
\begin{displaymath}
\begin{array}{cc}
\xymatrix{ X' \ar[r] \ar[d] & X \ar[d] \\ A' \ar[r] & A } &
\xymatrix{ Y' \ar[r] \ar[d] & Y \ar[d] \\ B' \ar[r] & B }
\end{array}
\end{displaymath}
are homotopy pullbacks in \ct{C}, then so is
\diag{ X' + Y' \ar[r] \ar[d] & X + Y \ar[d] \\ A'  + B' \ar[r] & A + B. }
\end{prop}
\begin{proof}
Let $P$ be such that
\diag{ P \ar[r] \ar[d] & X + Y \ar[d] \\ A'  + B' \ar[r] & A + B }
is a homotopy pullback. To show that $P$ is a homotopy sum of $X'$ and $Y'$ it suffices, by stability, to show that both
\begin{equation} \label{again2sqrs}
\begin{array}{cc}
\xymatrix{ X' \ar[r] \ar[d] & P \ar[d] \\
A' \ar[r] & A' + B' } &
\xymatrix{ Y' \ar[r] \ar[d] & P \ar[d] \\
B' \ar[r] & A' + B' }
\end{array}
\end{equation}
are homotopy pullbacks. To see this for the first square, note that we have a commuting cube
\diag{ & P \ar[rr] \ar'[d][dd] & & X + Y \ar[dd] \\
X' \ar[ur] \ar[rr] \ar[dd] & & X \ar[ur] \ar[dd] \\
& A' + B' \ar[rr] & & A + B. \\
A' \ar[ur] \ar[rr] & & A \ar[ur] }
Since the front, the back and the right face are homotopy pullbacks, the same holds for the left face. A similar cube shows that the second square in (\ref{again2sqrs}) is a homotopy pullback as well.
\end{proof}

\subsection{Homotopy exact completion.} If \ct{C} is a homotopy extensive path category, then ${\rm Hex}(\ct{C})$ will not only be exact: it will be a pretopos, that is, a category which both exact and extensive. This subsection will be devoted to a direct proof of this fact. (Alternatively, we could have appealed to \refprop{conntoordintheory} above and Section 3.4 in \cite{carbonivitale98}.)

\begin{prop}{hinitobjunderhex}
Homotopy initial objects become initial objects in the homotopy exact completion.
\end{prop}
\begin{proof}
Let $(X, R)$ be an arbitrary object in the homotopy exact completion with a fibration $\rho: R \to X \times X$. If $0$ is homotopy initial, there will be maps $f: 0 \to X$ and $g: 0 \to R$. Now $(f, f) \simeq \rho g$, so by \refprop{fillersuptohomotopy} there is also a map $g': 0 \to R$ such that $(f, f) = \rho g'$. Hence the square
\diag{ 0 \ar[r]^{g'} \ar[d]_r & R \ar[d]^{\rho} \\
P0 \ar[r]_(.4){(fs, ft)} & X \times X }
commutes. Since it has a weak equivalence on the left and a fibration on the right, there is a map $P0 \to R$ to track $f$, showing the existence of a morphism $(0, P0) \to (X, R)$ in the homotopy exact completion. To prove uniqueness, note that if there are two maps $f, f': 0 \to X$ then $(f, f') \simeq \rho g$, which shows that $f$ and $f'$ are identical as maps in the homotopy exact completion (see \refrema{onfillersinhex}).
\end{proof}

\begin{theo}{hsumsunderhex}
If \ct{C} is a homotopy extensive path category, then its homotopy exact completion ${\rm Hex}(\ct{C})$ is a pretopos.
\end{theo}
\begin{proof}
It will be convenient to use the first alternative description of ${\rm Hex}(\ct{C})$ in terms of pseudo-equivalence relations, as in \refprop{thirddescrofhex}. So let $R \to X \times X$ and $S \to Y \times Y$ be two pseudo-equivalence relations.

If $R + S$ and $X + Y$ are the homotopy sums, then from the maps $X \times X \to (X + Y) \times (X + Y)$ and $Y \times Y \to (X + Y) \times (X + Y)$ and the universal property of $R + S$ we obtain a map \[ R + S \to (X + Y) \times (X + Y). \]
Using the properties of homotopy extensive categories that we have established, one can show that this map is a pseudo-equivalence relation and indeed the sum of $R \to X \times X$ and $S \to Y \times Y$ in the homotopy exact completion. The (easy) verification that these sums are stable and disjoint is left to the reader.
\end{proof}

\begin{rema}{ongranvitale} \reftheo{hsumsunderhex} could also have been derived from results in \cite{granvitale98,lackvitale01}, but it turns out that it is not difficult to give a direct proof, so that is what we have done here.
\end{rema}

\subsection{Homotopy natural numbers object} A homotopy natural numbers object we define, like a homotopy sum, as a natural numbers object in the homotopy category.
\begin{defi}{hnno}
An object $\NN$ together with maps $0: 1 \to \NN$ and $\sigma: \NN \to \NN$ is a \emph{homotopy natural numbers object} (hnno) if for any pair of maps $y_0: 1 \to Y$ and $g: Y \to Y$ there is a map $h: \NN \to Y$, unique up to homotopy, such that $h 0 \simeq y_0$ and $h \sigma \simeq gh$.
\end{defi}

\begin{prop}{equivcharhnno}
An object $\NN$ together with maps $0: 1 \to \NN$ and $\sigma: \NN \to \NN$ is a homotopy natural numbers object if and only if for any commuting diagram of the form
\diag{ & X \ar[r]^f \ar[d]^p & X \ar[d]^p \\
1 \ar[r]_0 \ar[ur]^{x_0} & \NN \ar[r]_\sigma & \NN }
where $p$ is a fibration, there is a section $a: \NN \to X$ of $p$ such that $a 0 \simeq x_0$ and $a \sigma \simeq f a$.
\end{prop}
\begin{proof}
The argument is very similar to proofs of both \refprop{othercharhinitialobj} and \refprop{othercharhomsum}, so we will not give many details here. Let us just point out how one proves that if $0: 1 \to \NN$ and $\sigma: \NN \to \NN$ are as in the statement of the proposition, then for any pair of maps $y_0: 1 \to Y$ and $g: Y \to Y$ and for any pair of maps $h, h': \NN \to Y$ such that $h 0 \simeq y_0$ and $H: h \sigma \simeq gh$ and $h' 0 \simeq y_0$ and $K: gh' \simeq h' \sigma $, one must have $h \simeq h'$. For this one constructs the pullback
\diag{ X \ar[d]_p \ar[r]^\pi & PY \ar[d]^{(s, t)} \\
\NN \ar[r]_(.4){(h,h')} & Y \times Y, }
and considers the maps $Hp, Pg \circ \pi, Kp: X \to PY$, where $Pg: PY \to PY$ is a map such that $(s, t)Pg = (g \times g)(s, t)$. Since
\[ tHp = ghp = gs\pi = s (Pg) \pi \]
and
\[ t (Pg) \pi = gt\pi = gh'p = sKp, \]
we can use the composition operation on $PY$ to construct a map $L: X \to PY$ with $sL = sHp = h\sigma p$ and $tL = tKp = h'\sigma p$. Together with the universal property of $X$ this gives one a map $f: X \to X$ with $pf = \sigma p$ and $\pi f = L$. Since $h0 \simeq y_0 \simeq h'0$, there is also a map $x_0: 1 \to X$ with $px_0 = 0$. From this follows that $p$ has a section and hence that $h$ and $h'$ are homotopic.
\end{proof}

\begin{prop}{hexandhnno}
If \ct{C} is a path category with a homotopy natural numbers object, then ${\rm Hex}(\ct{C})$ has a natural numbers object.
\end{prop}
\begin{proof}
Suppose that $\NN$ is an object in \ct{C} which comes equipped with maps $0: 1 \to \NN$ and $\sigma: \NN \to \NN$ having the property as in the previous proposition. It is not hard to see that $i\NN = (\NN, P\NN)$ must have the same property in ${\rm Ex}(\ct{C})$, and therefore it becomes a natural numbers object in its homotopy category ${\rm Hex}(\ct{C})$.
\end{proof}

\section{$\Pi$-types}

In this section, we study a suitable notion of function space in path categories and the structure these function spaces induce on the homotopy exact completion.
We are guided by the relevant properties of the classical exact completion of categories with finite limits, where the existence of a weak kind of internal hom-object in every slice of the original category implies that every slice of the exact completion has actual internal homs, i.e., is a locally cartesian closed category \cite{carbonirosolini00}.  These weak internal hom-objects enjoy the existence condition for the internal hom in the sense that any map $A \times B \to C$ gives a map $A \to {\rm Hom}(B,C)$, but the latter is not required to be unique.  A similar situation arises in type theory, and the path categories constructed as syntactic categories of dependent type theories only possess such weak internal homs. It is important to realise that for these type-theoretic categories there is \emph{a priori} no uniqueness condition involved at all, not even in a up-to-homotopy sense. (Uniqueness up to homotopy is related to an additional property of type theory called function extensionality, see \refrema{functionextensionality}  below.)

More generally, dependent type theories usually include a type constructor for $\Pi$-types. For a path category \ct{C} arising as the syntactic category of such
a type theory, the pullback functors $\ct{C}(B) \to \ct{C}(A)$ along fibrations $B \to A$ have a weak kind of right adjoint (weakness here is meant in the same sense as for internal homs above).

In this section, we will define notions of weak homotopy exponential and weak homotopy $\Pi$-type in the context of an arbitrary path category \ct{C}. These notions are chosen in such a way that for the special case where \ct{C} is a category with finite limits and every map in \ct{C} is a fibration and every weak equivalence in \ct{C} is an isomorphism, having weak homotopy $\Pi$-types corresponds to the notion of weak local cartesian closure from \cite{carbonirosolini00}. In addition, these notions are sufficiently weak to ensure that these structures exist in the syntactic path category obtained from a type theory possessing the corresponding type constructions, even if in the type theory the computation rules would hold only in a propositional form. Finally, the notion of weak homotopy $\Pi$-type is sufficiently strong to ensure that the homotopy exact completion ${\rm Hex}(\ct{C})$ is locally cartesian closed if \ct{C} has weak homotopy $\Pi$-types.

\subsection{Definition and properties.}
Throughout this section \ct{C} will be a path category.
\begin{defi}{hexponential}
For objects $X$ and $Y$ in \ct{C} a \emph{weak homotopy exponential} is an object $X^Y$ together with a map ${\rm ev}: X^Y \times Y \to X$ such that for any map $h: A \times Y \to X$ there is a map $H: A \to X^Y$ such that
\diag{ X^Y \times Y \ar[r]^(.6){\rm{ev}} & X \\
A \times Y \ar[ur]_h \ar[u]^{H \times 1} }
commutes up to homotopy. If such a map $H$ is unique up to homotopy, then $X^Y$ is a \emph{homotopy exponential}.
\end{defi}

\begin{defi}{hPitypes}
The category \ct{C} has \emph{weak homotopy $\Pi$-types} if for any two fibrations $f: X \to J$ and $\alpha: J \to I$ there is a an object $\Pi_\alpha X = \Pi_\alpha f$ in $\ct{C}(I)$, that is, a fibration $\Pi_\alpha X \to I$, together with an evaluation map ${\rm ev}: \alpha^* \Pi_\alpha X \to X$ over $J$, with the following weak universal property: if there are maps $g: Y \to I$ and $m: \alpha^* Y \to X$ with $m$ over $J$, then there exists a map $n: Y \to \Pi_\alpha X$ over $I$ such that $m: \alpha^* Y \to X$ and ${\rm ev} \circ \alpha^* n: \alpha^* Y \to X$ are fibrewise homotopic over $J$. If the map $n$ is unique with this property up to fibrewise homotopy over $I$, we call $\Pi_\alpha f$ and ${\rm ev}: \alpha^* \Pi_\alpha X \to X$ a \emph{homotopy $\Pi$-type}.
\end{defi}

\begin{rema}{altdefhPitypes} We will not need this observation, but we would like to point out that in the definition above it is sufficient to consider only fibrations $g: Y \to I$.
\end{rema}

In the proofs of the following two propositions we only give the constructions: verifications are left to the reader.

\begin{prop}{hPiimpliesslicesexp}
If \ct{C} has (weak) homotopy $\Pi$-types then each $\ct{C}(I)$ has (weak) homotopy exponentials.
\end{prop}
\begin{proof}
Given $Y, Z \in \ct{C}(I)$ one defines $Z^Y$ in $\ct{C}(I)$ as $\Pi_\alpha (\pi_2)$, where $\alpha: Y \to I$ and $\pi_2: Z \times_I Y \to Y$.
\end{proof}

\begin{prop}{expnniceifPi}
Let \ct{C} be a path category with (weak) homotopy $\Pi$-types. Given a fibration $p: Z \to Y$ and a (weak) homotopy exponential $(Y^X, {\rm ev})$, there is a (weak) homotopy exponential $(Z^X, {\rm ev})$ and a fibration $p^X: Z^X \to Y^X$ such that
\begin{enumerate}
\item[(i)] The diagram
\diag{ Z^X \times X \ar[r]^(.6){\rm ev} \ar[d]_{p^X \times X} & Z \ar[d]^p \\
Y^X \times X \ar[r]_(.6){\rm ev} & Y }
commutes.
\item[(ii)] For each $T$ the diagram
\diag{ {\rm Ho}(\ct{C})(T, Z^X) \ar[d] \ar[r] & {\rm Ho}(\ct{C})(T \times X, Z) \ar[d] \\
{\rm Ho}(\ct{C})(T, Y^X) \ar[r] & {\rm Ho}(\ct{C})(T \times X, Y)}
in $\Sets$ has the property that the map from $ {\rm Ho}(\ct{C})(T, Z^X)$ to the inscribed pullback is an isomorphism in case $Z^X$ is a homotopy exponential, and an epimorphism in case $Z^X$ is a weak homotopy exponential.
\end{enumerate}
\end{prop}
\begin{proof}
Given $Y^X$ with its evaluation ${\rm ev}: Y^X \times X \to Y$ let $q$ be the pullback
\diag{ P \ar[r] \ar[d]_q & Z \ar[d]^p \\
Y^X \times X \ar[r]_(.6){\rm ev} & Y,}
and let $Z^X$ be $\Pi_{\pi_1}(q)$, where $\pi_1: Y^X \times X \to Y^X$.
\end{proof}

\begin{coro}{Piandsections}
Suppose $p: Z \to Y$ is a fibration and $p^X: Z^X \to Y^X$ is the fibration obtained from it as in the previous proposition. Then any section $s: Y \to Z$ induces a section $s^X$ of $p^X$ such that
\diag{ Y^X \times X \ar[r]^(.6){\rm ev} \ar[d]_{s^X \times 1_X} & Y \ar[d]^s \\
Z^X \times X \ar[r]_(.6){\rm ev} & Z }
commutes up to homotopy.
\end{coro}
\begin{proof}
Consider the diagram in (ii) in the previous proposition with $T = Y^X$. Using that the map to the inscribed pullback is an epimorphism, one finds a map $\sigma: Y^X \to Z^X$ that upon postcomposition with $p^X$ is homotopic to the identity and such that
 \[ {\rm ev} \circ (\sigma \times 1_X) \simeq s \circ {\rm ev}. \]
Using that $p^X$ is a fibration, one may replace $\sigma$ by a homotopic map $s^X$ such that $p^Xs^X = 1$ and ${\rm ev} \circ (s^X \times 1_X) \simeq s \circ {\rm ev}$.
\end{proof}

\begin{prop}{onfunctionextensionality}
Suppose \ct{C} is a path category with weak homotopy $\Pi$-types, and let $X^Y$ be a weak homotopy exponential in \ct{C}. We may choose $X^Y \times X^Y$ as a suitable weak homotopy exponential $(X \times X)^Y$ and choose $(PX)^Y$ as in \refprop{expnniceifPi}, so that $(s^Y, t^Y): (PX)^Y \to X^Y \times X^Y$ is a fibration. Then the following are equivalent:
\begin{enumerate}
\item $X^Y$ is a homotopy exponential.
\item There is a morphism $e: (PX)^Y \to P(X^Y)$ such that $(s, t)e = (s^Y, t^Y)$.
\end{enumerate}
Also, if both $X^Y$ and $(PX)^Y$ are homotopy exponentials, then the canonical map $P(X^Y) \to (PX)^Y$ is a homotopy equivalence.
\end{prop}
\begin{proof} (1) $\Rightarrow$ (2): The diagrams
\begin{displaymath}
\begin{array}{cc}
\xymatrix{ (PX)^Y \times Y \ar[r]^(.6){\rm ev} \ar[d]_{s^Y \times Y} & PX \ar[d]^s \\
X^Y \times Y \ar[r]_(.6){\rm ev} & X } &
\xymatrix{ (PX)^Y \times Y \ar[r]^(.6){\rm ev} \ar[d]_{t^Y \times Y} & PX \ar[d]^t \\
X^Y \times Y \ar[r]_(.6){\rm ev} & X }
\end{array}
\end{displaymath}
commute, while $s\simeq t$ and $s \circ {\rm ev} \simeq t \circ {\rm ev}$. So if $X^Y$ is a homotopy exponential, the maps $s^Y$ and $t^Y$ must be homotopic. Therefore there exists a map $e: (PX)^Y \to P(X^Y)$ such that $s e = s^Y$ and $t  e = t^Y$.

(2) $\Rightarrow$ (1):  Suppose that there is a map $h: A \times Y \to X$ together with morphisms $H_1, H_2: A \to X^Y$ such that ${\rm ev}(H_1 \times Y) \simeq h \simeq {\rm ev}(H_2 \times Y)$. The latter means that there is a map $K: A \times Y \to PX$ such that
\[ (s, t)K = ({\rm ev}(H_1 \times Y), {\rm ev}(H_2 \times Y)). \]
Since
\diag{ {\rm Ho}(\ct{C})(A, (PX)^Y) \ar[d] \ar[r] & {\rm Ho}(\ct{C})(A \times Y, PX) \ar[d] \\
{\rm Ho}(\ct{C})(A, X^Y \times X^Y) \ar[r] & {\rm Ho}(\ct{C})(A \times Y, X \times X)}
is a quasi-pullback, there is a map $L: A \to (PX)^Y$ such that $(s^Y, t^Y) L = (H_1,H_2)$ and ${\rm ev} \circ (L \times Y) \simeq K$. So if there is a map $e: (PX)^Y \to P(X^Y)$ such that $(s, t)e = (s^Y, t^Y)$, then for $M := eL$ we have $(s, t)M = (H_1, H_2)$, showing that $H_1$ and $H_2$ are homotopic.

Finally, note that there is always a morphism $P(X^Y) \to (PX)^Y$ making the lower triangle in
\diag{ X^Y \ar[r]^{r^Y} \ar[d]_r & (PX)^Y \ar[d]^{(s^Y, t^Y)} \\
P(X^Y) \ar[r]_{(s,t)} \ar@{..>}[ur] & X^Y \times X^Y }
commute, whilst making the upper triangle commute up to homotopy. We claim that if both $X^Y$ and $(PX)^Y$ are homotopy exponentials, then this diagonal arrow is a homotopy equivalence. For this it suffices to prove that the arrow along the top is a homotopy equivalence, because the arrow on the left is and the upper triangle commutes up to homotopy. To see that $r^Y$ is a homotopy equivalence, observe that $r: X \to PX$ is a homotopy equivalence and therefore an isomorphism in the homotopy category. So if $X^Y$ and $(PX)^Y$ are exponentials in the homotopy category, then $r^Y$ is an isomorphism in the homotopy category, that is, a homotopy equivalence.
\end{proof}

\begin{rema}{functionextensionality} What the preceding proposition shows is that ordinary homotopy exponentials are those weak homotopy exponentials that satisfy what type-theorists call function extensionality (indeed, in the syntactic category the morphism $e$ would be a proof term for the type-theoretic translation of the statement that two functions $f, g: Y \to X$ are equal if $f(y)$ and $g(y)$ are equal for every $y \in Y$). This principle is not valid in the syntactic category associated to type theory, and for this reason the homotopy exponentials in the syntactic category are only weak. The same applies to the homotopy $\Pi$-types that we have defined: the syntactic category only has these in the weak form.
\end{rema}

\subsection{Homotopy exponentials and homotopy exact completion} The main goal of this section is to show that ${\rm Hex}(\ct{C})$ is locally cartesian closed, whenever \ct{C} has weak homotopy $\Pi$-types. We will only outline the constructions here, as a detailed verification that they indeed have the required properties is both straightforward and cumbersome.

\begin{prop}{hexpunderhex}
If \ct{C} has weak homotopy $\Pi$-types, then ${\rm Ex}(\ct{C})$ has homotopy exponentials and ${\rm Hex}(\ct{C})$ has ordinary exponentials.
\end{prop}
\begin{proof}
Assume \ct{C} has weak homotopy $\Pi$-types, and let $(X, R)$ and $(Y, S)$ be two objects in ${\rm Hex}(\ct{C})$; our goal is to construct the exponential $(X, R)^{(Y, S)}$.

The idea is to take $(W, Q)$ where $W$ is the pullback:
\diag{ W \ar[rr] \ar[d]_p & & R^S \ar[d] \\
X^Y \ar[r]_(.3){\delta} & (X \times X)^{Y \times Y} \ar[r] & (X \times X)^S. }
Here $\delta$ is a map making
\diag{ X^Y \times Y \times Y \ar[rr]^(.4){\delta \times 1} \ar[drr]_{\epsilon} & & (X \times X)^{Y \times Y} \times (Y \times Y) \ar[d]^{\rm ev} \\
& & X \times X }
commute up to homotopy with $\epsilon = ({\rm ev}(p_1,p_2),{\rm ev}(p_1,p_3))$, while the map $R^S \to (X \times X)^S$ has the properties from \refprop{expnniceifPi}; in particular it is a fibration and the pullback $W$ does indeed exist. The object $Q$ is obtained as the pullback
\diag{ Q \ar[r] \ar[d] & R^Y \ar[d] \\
W \times W \ar[r]_(.3){p \times p} & X^Y \times X^Y \cong (X \times X)^Y,}
where we have used that $X^Y \times X^Y$ acts as a suitable weak homotopy exponential $(X \times X)^Y$. In addition, the map on the right is built in accordance with  \refprop{expnniceifPi}; this means in particular that it is a homotopy equivalence relation and therefore the same is true for $Q \to W \times W$. We leave it to the reader to verify that $(W, Q)$ is indeed an exponential in ${\rm Hex}(\ct{C})$ and a homotopy exponential in ${\rm Ex}(\ct{C})$.
\end{proof}

In fact, we can even prove that ${\rm Hex}(\ct{C})$ is locally cartesian closed whenever the path category \ct{C} has weak homotopy $\Pi$-types. Recall that a category with finite limits is locally cartesian closed if every slice is cartesian closed. For this one sometimes only needs to verify that slices over certain objects are cartesian closed. For instance, if \ct{E} is exact and $I$ is an object in \ct{E} fitting into a coequalizer diagram
\diag{ Q \ar@<1ex>[r] \ar@<-1ex>[r]  & P \ar[r] & I, }
where $Q \to P \times P$ is a pseudo-equivalence relation, then an exponential $(X \to I)^{(Y \to I)}$ in $\ct{E}/I$ may be computed from two exponentials in $\ct{E}/Q$ and $\ct{E}/P$ by taking the coequalizer of the two parallel arrows along the top in the diagram below:
\diag{ \big( (X \times_I Q)^{(Y \times_I Q)} \big)_Q \ar@<1ex>[r] \ar@<-1ex>[r] \ar[d] & \big( (X \times_I P)^{(Y \times_I P)} \big)_P \ar[d] \\
Q \ar@<1ex>[r] \ar@<-1ex>[r]  & P.}
(This is called the \emph{method of descent}, for which exactness of \ct{E} is crucial.)

\begin{theo}{hpiunderhex}
Let \ct{C} be a path category. If \ct{C} has weak homotopy $\Pi$-types, then ${\rm Hex}(\ct{C})$ is locally cartesian closed.
\end{theo}
\begin{proof} We need to prove that each slice category ${\rm Hex}(\ct{C})/I$ has exponentials. For this it suffices to consider the case where $I = iZ$: any object in ${\rm Hex}(\ct{C})$ is covered by such an object (see \refprop{iiscovering}), so the general case follows by descent.

We have proved in \reftheo{equivwithstableelements} that ${\rm Hex}(\ct{C})/iZ$ is equivalent to the full subcategory of ${\rm Hex}(\ct{C}(Z))$ on the stable objects. It follows from \refprop{hPiimpliesslicesexp} and \refprop{hexpunderhex} that ${\rm Hex}(\ct{C}(Z))$ has exponentials, so it suffices to prove that if we take an exponential of two stable objects in this category, then the result is again stable.

So let $(f: X \to Z, \rho: R \to X \times_Z X)$ and $(g: Y \to Z, \sigma: S \to Y \times_Z Y)$ be two stable objects in ${\rm Hex}(\ct{C}(Z))$. These will have two (essentially unique) transport structures $\Gamma_X: X \times_Z PZ \to X$ and $\Gamma_Y: Y \times_Z PZ \to Y$; recall that stability means that $\Gamma(x, \alpha) \simeq_R x$ and $\Gamma(y, \alpha) \simeq_S y$ whenever $\alpha$ is a loop in $Z$.

So let $(h: W \to Z, q: Q \to W \times_Z W)$ be the result of computing the exponential $(X, R)^{(Y, S)}$ over $(Z, PZ)$ as in the previous proposition. This object has a transport structure as well, which is probably best described in words. What this action should do is to associate to every $w \in W$ living over $z \in Z$ and path $\alpha$ from $z$ to $z'$ a new element $w' \in W$ over $z'$. Such a $w'$ is intuitively a function, so let $y' \in Y$ be an element over $z'$. We can transport $y'$ back along the inverse of $\alpha$ to an element $y$ over $z$; to this $y$ we can apply $w$ and obtain an element $x$ over $z$. Using transport again, but now on $x \in X$ and $\alpha$ we find an element $x' \in X$ over $z'$. The idea is to set $w'$ to be the function sending $y'$ to $x'$.  \refprop{TtransppresTequivalence} implies that $w'$ will be tracked whenever $w$ is.

If $\alpha$ is a loop, then $y'$ would be $S$-equivalent to $y$ and $x$ would be $T$-equivalent to $x'$. This means that $w$ and $w'$ would be $Q$-equivalent, showing that $(W, Q)$ is stable, as desired.
\end{proof}

\begin{rema}{onexcomplofcatwweakfinitelimits} One could also have derived \reftheo{hpiunderhex} from \refprop{conntoordintheory} above and the results in \cite{carbonirosolini00}. We have included a direct proof of \reftheo{hpiunderhex} here, because it provides a description of the exponentials in slices of ${\rm Hex}(\ct{C})$  which only makes sense in the specific context of exact completions of path categories and would not work in the more general context of exact completions of categories with weak finite limits.

In addition, these constructions can also be used to show that ${\rm Ex}(\ct{C})$ has homotopy $\Pi$-types whenever \ct{C} has weak homotopy $\Pi$-types. In view of \refrema{functionextensionality}  this means that ${\rm Ex}(\ct{C})$ satisfies a form of function extensionality even when \ct{C} does not.
\end{rema}

\bibliographystyle{plain} \bibliography{hSetoids}

\end{document}